\documentclass[a4paper,11pt,reqno]{amsart}
\usepackage{amssymb,amsthm,amsmath}
\usepackage{mathtools}
\usepackage{amsfonts}
\usepackage{amscd}
\usepackage{amssymb}
\usepackage{enumerate}
\usepackage{bbm}
\usepackage{graphicx}
\allowdisplaybreaks
\usepackage{color}
\usepackage{amsbsy}
\usepackage{graphicx}
\usepackage{amsthm}
\usepackage{amsmath}
\usepackage{amsxtra}
\usepackage{mathrsfs}
\usepackage{bbm}
\usepackage{dsfont}
\usepackage{enumitem}
\usepackage{comment}
\usepackage{xcolor}
\usepackage{float} 
\usepackage{url}
\usepackage{soul}
\usepackage[hidelinks]{hyperref}

\usepackage{ifthen}

\hypersetup{
    colorlinks=true,
    linkcolor=red,      
    citecolor=red,       
    urlcolor=cyan        
}

\newtheorem{theorem}{Theorem}[section]
\newtheorem{defi}[theorem]{Definition}
\newtheorem{remark}[theorem]{Remark}
\newtheorem{corollary}[theorem]{Corollary}
\newtheorem{prop}[theorem]{Proposition}
\newtheorem{lemma}[theorem]{Lemma}

\numberwithin{equation}{section}

\definecolor{red}{rgb}{1.0, 0.0, 0.0}
\newcommand{\Bea}{\begin{eqnarray*}}
    \newcommand{\Eea}{\end{eqnarray*}}
\newcommand{\Be} {\begin{equation*}}
    \newcommand{\Ee} {\end{equation*}}
\newcommand{\be} {\begin{equation}}
    \newcommand{\ee} {\end{equation}}
\newcommand{\bea} {\begin{eqnarray}}
    \newcommand{\eea} {\end{eqnarray}}

\usepackage{color}

\title[Uncertainty Principles for the STFT on the Heisenberg Group]{Uncertainty Principles for the Short-Time Fourier Transform on the Heisenberg Group}
\author[A.~Dabra]{Arvish Dabra}
\address{
    Arvish Dabra:
    \endgraf
    Department of Mathematics
    \endgraf
    Indian Institute of Science Education and Research Bhopal
    \endgraf
    Bhopal, Madhya Pradesh - 462066
    \endgraf
    India
    \endgraf
    {\it E-mail address:} {\rm arvishdabra3@gmail.com}
}

\author[A.~Dasgupta]{Aparajita Dasgupta}
\address{
    Aparajita Dasgupta:
    \endgraf
    Department of Mathematics
    \endgraf
    Indian Institute of Technology Delhi
    \endgraf
    Hauz Khas, New Delhi - 110016
    \endgraf
    India
    \endgraf
    {\it E-mail address:} {\rm adasgupta@maths.iitd.ac.in}
}

\author[P.~Gulia]{Prerna Gulia}
\address{
    Prerna Gulia:
    \endgraf
    Department of Mathematics
    \endgraf
    Indian Institute of Technology Delhi
    \endgraf
    Hauz Khas, New Delhi - 110016
    \endgraf
    India
    \endgraf
    {\it E-mail address:} {\rm prernagulia64@gmail.com}
}


\vspace{1cm}

\allowdisplaybreaks

\begin{document}

\begin{abstract}

We develop a systematic theory of uncertainty principles for the short-time Fourier transform (STFT) on the Heisenberg group. Building on recent developments in modulation spaces and time-frequency analysis on the Heisenberg group introduced by Fischer et al.~and later by Biswas--Thangavelu, we establish noncommutative analogues of several fundamental uncertainty principles in time-frequency analysis, including Benedicks’ theorem, the Donoho--Stark uncertainty principle, and Lieb’s inequality. As a consequence of Lieb's inequality, we derive an entropy-based uncertainty principle of Hirschman type. We further establish a Heisenberg-type uncertainty inequality and local uncertainty principles in the spirit of Price. In addition, we prove a Beurling--Hardy-type theorem that captures the interplay between decay and phase-space localization. Finally, we investigate decay properties of the STFT and their implications for time-frequency concentration. These results extend a broad spectrum of classical uncertainty phenomena from the Euclidean setting to the Heisenberg group, highlighting the role of noncommutative harmonic analysis in the study of phase-space localization.

\end{abstract}

\maketitle

\section{Introduction}

Uncertainty principles occupy a central position in harmonic analysis, capturing the fundamental obstruction to the simultaneous localization of a function and its transform. This phenomenon manifests itself in various forms and reflects an intrinsic trade-off between spatial and frequency localization, where concentration in one domain necessarily limits concentration in the other. The classical Heisenberg uncertainty principle provides the fundamental formulation of this phenomenon. For a sufficiently regular function $f$ on $\mathbb{R}^n$, it quantifies the reciprocal relationship between the spatial spread of $f$ and the frequency spread of its Fourier transform $\mathcal{F}(f)$. This fundamental principle has subsequently evolved into a rich theory encompassing qualitative, quantitative, entropic, weighted, and decay-based uncertainty principles.

Since its classical formulation, the uncertainty principle has developed in numerous directions. Different notions of concentration and localization have led to weighted uncertainty inequalities, entropic uncertainty principles \cite{Hirschman}, qualitative uncertainty principles, support-type results, and uncertainty principles formulated in terms of decay conditions \cite{gevrey}. A comprehensive overview of various forms of uncertainty principles can be found in the survey by Folland and Sitaram \cite{FSsurvey}. Over the past several decades, uncertainty principles have also been studied in a wide variety of mathematical settings, including different integral transforms, function spaces, and non-Euclidean structures. Given the extensive literature on the subject, it is impossible to cite all relevant contributions; we refer the reader to \cite{Ciatti-Ricci-Sundari-uncertainity_on_stratified,dabra2026uncertainty,jaming-Ghobber-bessel_transform, Price-Sitaram-Local_uncertainty,Smith-uncertainty_on_lct} and the references therein.

In signal analysis and quantum mechanics, the study of uncertainty is naturally connected with representations that describe a signal simultaneously in the time and frequency domains. Two important examples are the short-time Fourier transform (STFT) and the Wigner distribution. These representations associate with a function on $\mathbb{R}^n$ a function on the phase space \(\mathbb{R}^n\times\mathbb{R}^n\), thereby providing information about the distribution of a signal across both time and frequency. This phase-space perspective provides a natural framework for formulating and studying uncertainty principles. An illuminating motivation for time-frequency analysis, based on an analogy with a musical score, can be found in \cite{Foundations-of-time-frequency-analysis}.

The short-time Fourier transform on \(\mathbb{R}^n\), defined using modulation and Euclidean translation, has been extensively studied in connection with uncertainty principles, see \cite{bonami,feichtinger-advances_gabor_analysis,Foundations-of-time-frequency-analysis,Poria-uncertainty_math_phy,wilczok-uncertainty}. In particular, various classical uncertainty principles, including the Donoho--Stark principle \cite{Donoho-stark_paper} and Benedicks' theorem \cite{benedicks-original-paper}, admit natural analogues in the time-frequency setting. More recently, Poria and the second author \cite{poriadasgupta2024} investigated uncertainty principles for the short-time Fourier transform on the lattice \(\mathbb{Z}^n \times \mathbb{T}^n\), establishing analogues of the Donoho--Stark uncertainty principle, Benedick's theorem and several related results.

The short-time Fourier transform also plays a central role in the theory of modulation spaces. Introduced by Feichtinger \cite{Feichtinger-modulation_spaces} in 1983, modulation spaces have become an important class of function spaces in time-frequency analysis. They provide a natural framework for measuring the time-frequency concentration of a function through integrability conditions imposed on its short-time Fourier transform. More precisely, for a suitable nonzero window $g$, the modulation space $M^{p,q}(\mathbb{R}^n)$ consists of those functions $f$ for which the STFT $V_gf$ belongs to the mixed-norm space $L^{p,q}(\mathbb{R}^{2n})$. Thus, modulation spaces encode information about the joint behavior of a function in the spatial and frequency variables, rather than treating these aspects independently. The theory of modulation spaces has developed extensively and has deep connections with several important areas of analysis, including Besov spaces, atomic decompositions, and Gabor analysis, see \cite{Feichtinger-atomic_modulation_spaces,Feichtinger-grochenig-gabor_frames,Frazier-decomposition_besov_spaces}.

A natural question is whether the time-frequency framework underlying the classical short-time Fourier transform, together with its associated uncertainty principles, extends to non-Euclidean settings. The Heisenberg group $\mathbb{H}^n$ provides a particularly natural setting for such an extension, owing to its central role in noncommutative harmonic analysis and its close connection with time-frequency analysis. Fischer, Rottensteiner, and Ruzhansky~\cite{rottensteiner-heisenberg_modulation_space} introduced the Heisenberg-modulation spaces $E^{p,q}(\mathbb{R}^{2n+1})$ through the matrix coefficients of a square-integrable projective representation of the Dynin--Folland group. More recently, Biswas and Thangavelu~\cite{Biswas-Thangavelu} proposed a different construction of modulation spaces on $\mathbb{H}^n$, based on twisted modulation spaces and representations of another nilpotent Lie group. The present work is based on the former framework and investigates uncertainty principles for the short-time Fourier transform arising from the Dynin--Folland representation.

The Dynin--Folland group $\mathbb{H}_{n,2}$ originates in the work of Dynin~\cite{Dynin}, who developed a Weyl quantization of $\mathbb{H}^n$ using the representation theory of this group. It is a three-step nilpotent Lie group that may be realized as a semidirect product of $\mathbb{R}^{2n+2}$ and $\mathbb{H}^n$. Its Lie algebra is generated by the left-invariant vector fields on the Heisenberg group, together with the operators of multiplication by its $(2n+1)$ coordinate functions. Folland subsequently highlighted this construction in his monograph~\cite[p.~90]{folland-harmonic_phase_space}, observing that $\mathbb{H}_{n,2}$ may naturally be regarded as ``the Heisenberg group of the Heisenberg group.'' This viewpoint was developed further in~\cite{folland-meta_heisenberg}, where Folland systematically investigated analogous ``meta-Heisenberg groups'' associated with two-step nilpotent groups. The terminology \emph{Dynin--Folland group} recognizes Dynin's original construction and Folland's subsequent development of its structural theory. 

The representation-theoretic origin of the transform considered here is essential to its structure. If $\pi^{\mathrm{pr}}$ denotes the relevant projective representation of $\mathbb{H}_{n,2}$ and
\[
\check{\pi}^{\mathrm{pr}}
    = \mathcal{F}^{-1}\pi^{\mathrm{pr}}\mathcal{F},
\]
then the short-time Fourier transform of $f$ with respect to a window $g$ is given by
\[
V_gf(Q,P)
    = \left\langle f,\check{\pi}^{\mathrm{pr}}(Q,P)g\right\rangle
    = \left\langle \mathcal{F}(f),
      M_QT_P\mathcal{F}(g)\right\rangle,
    \qquad Q,P\in\mathbb{H}^n,
\]
where
\[
T_Ph(X)=h(X\cdot P)
\]
is the right translation with respect to the Heisenberg group law, while
\[
M_Qh(X)=e^{2\pi i\langle Q,X\rangle}h(X)
\]
is the Euclidean modulation under the identification
$\mathbb{H}^n\simeq\mathbb{R}^{2n+1}$. Thus, unlike the classical STFT on $\mathbb{R}^d$, the transform studied here combines Euclidean frequency modulation with noncommutative spatial translation. Although its phase space is $\mathbb{H}^n\times\mathbb{H}^n$, the two variables play analytically different roles.

This hybrid structure is not merely a formal distinction, it determines the analytical methods used throughout the paper. For a fixed translation parameter, the transform can be analyzed using the Euclidean Fourier transform in the modulation variable. Integration with respect to the translation parameter, however, leads naturally to convolution on the Heisenberg group. This interaction is particularly evident in the proof of the Lieb estimate, which combines the sharp Hausdorff--Young inequality on $\mathbb{R}^{2n+1}$ with the sharp Young convolution inequality on $\mathbb{H}^n$. Consequently, the classical Euclidean arguments do not transfer verbatim, rather, they must be reorganized to accommodate simultaneously the Euclidean frequency structure and the noncommutative group law.

To the best of our knowledge, a systematic theory of uncertainty principles has not previously been developed for this short-time Fourier transform on the Heisenberg group. The purpose of the present paper is to provide such a theory by treating qualitative, quantitative, entropic, moment-based, local, and decay-based forms of uncertainty principles within a common representation-theoretic framework. Our principal results may be summarized as follows.
\begin{enumerate}
\item We establish the basic properties of the STFT on $\mathbb{H}^n$, including its orthogonality relation, Plancherel formula, inversion formula, and reproducing kernel structure.

\item We prove qualitative uncertainty principles of two different types. First, we obtain a Benedicks-type support theorem: if both $\operatorname{supp}(V_gf)$ and $\operatorname{supp}(\mathcal{F}(g))$ have finite measure, then either $f=0$ or $g=0$. We also establish a Beurling--Hardy-type theorem in which a Beurling-type integrability condition on $V_gf$ is combined with Gaussian decay of $\mathcal{F}(f)$.

\item We derive quantitative concentration estimates of Donoho--Stark type and prove the Lieb inequality
\[
\int_{\mathbb{H}^n}\int_{\mathbb{H}^n}
   |V_gf(Q,P)|^p\,dQ\,dP
   \leq
   \left(\frac{2}{p}\right)^{2n+1}
   \bigl(\|f\|_2\|g\|_2\bigr)^p,
   \qquad p>2.
\]
This estimate yields a refined lower bound for the measure of any set on which the STFT is substantially concentrated.

\item As consequences of the preceding estimates, we obtain a Hirschman entropy uncertainty principle and Heisenberg-type lower bounds for weighted $L^2$-moments of the STFT. We further establish local uncertainty principles in the spirit of Price. These include estimates in both ranges $0 < \alpha < \mathcal{Q}/2$ and $\alpha>\mathcal{Q}/2$, where $\mathcal{Q}=2n+2$ is the homogeneous dimension of $\mathbb{H}^n$.

\item Finally, we investigate the decay properties of the STFT. We show that extended-Gevrey-type derivative bounds for $\mathcal{F}(f)$ and $\mathcal{F}(g)$ yield estimates for arbitrary polynomial weights in the modulation variable, uniformly when the translation parameter ranges over bounded subsets of $\mathbb{H}^n$.
\end{enumerate}

Taken together, these results show that the principal forms of classical time-frequency uncertainty persist in the Heisenberg setting, but their formulation and proofs reflect the hybrid geometry of the Dynin--Folland representation. The main contribution of the paper is therefore not simply a collection of separate analogues, but rather a unified theory of uncertainty in which support, concentration, entropy, moments, local behavior, and decay are connected through the same noncommutative time-frequency transform.

This paper is organized as follows. In Section \ref{sec 2}, we recall the necessary background on the Heisenberg and Dynin--Folland groups, introduce the short-time Fourier transform associated with the relevant projective representation, and establish its fundamental properties. Section \ref{sec Benedicks} contains an analogue of Benedicks’ theorem. Sections \ref{sec:donoho-stark} and \ref{sec Lieb} are devoted to the Donoho--Stark and Lieb uncertainty principles, respectively. In Section \ref{sec Hirschman}, we derive a Hirschman entropy inequality as a consequence of the Lieb estimate. In Section \ref{sec BeurlingHardy}, we establish a Beurling--Hardy-type uncertainty principle. In Section \ref{sec Heisenberg}, we prove Heisenberg-type inequality and local uncertainty principles of Price. Finally, Section \ref{sec:decay} investigates the decay properties of the short-time Fourier transform under extended-Gevrey-type derivative assumptions.\\

Throughout this article, the notation $A \lesssim B$ means that there exists a constant $c>0,$ depending only on the relevant parameters, such that $A \leq cB$.


\section{Preliminaries}\label{sec 2}

As discussed in the introduction, Fischer et al.~\cite{rottensteiner-heisenberg_modulation_space} introduced and studied the weighted modulation spaces $E^{p,q}_s(\mathbb{R}^{2n+1})$ associated with the Heisenberg group. These spaces are defined in terms of the matrix coefficients of a square-integrable representation of the Dynin--Folland group $\mathbb{H}_{n,2}$. The construction is analogous to that of the classical modulation spaces $M^{p,q}(\mathbb{R}^n)$, which are defined via the Schr\"odinger representation of the Heisenberg group $\mathbb{H}^n$. Their approach extends the classical theory by exploiting the representation-theoretic structure of the Dynin--Folland group and provides a natural framework for the study of time-frequency analysis on the Heisenberg group.

In this section, we recall the basic definitions and properties of the Heisenberg group, its Schr\"odinger representations and the Dynin--Folland group. We also discuss the square-integrable representation of the Dynin--Folland group that underlies the definition of the modulation spaces $E^{p,q}(\mathbb{R}^{2n+1})$.

\subsection{The Heisenberg group}
Let \(\mathbb{H}^n = \mathbb{C}^n \times \mathbb{R} \cong \mathbb{R}^n \times \mathbb{R}^n \times \mathbb{R}\) denote the \((2n+1)\)-dimensional Heisenberg group, equipped with the group law 
$$(p,q,t)(p',q',t')=\left(p+p',q+q',t+t'+\frac{1}{2}(pq'-qp') \right),\;\forall \, (p,q,t),(p',q',t') \in \mathbb{H}^n,$$
where $pq$ is the usual Euclidean dot product of $p,q \in \mathbb{R}^n$. The Heisenberg group is a two-step nilpotent and noncommutative Lie group whose well-developed representation theory plays a fundamental role in harmonic analysis and time-frequency analysis. For detailed expositions of the representation theory of $\mathbb{H}^n$, we refer the reader to \cite{fisher-ruzhansky-book, rottensteiner-thesis, rottensteiner-ruzhansky-harmonic-oscillators, thangaveluheisenberg}. 

\noindent We denote by \(\mathfrak{h}_n,\) the Heisenberg algebra whose basis consists of the left-invariant vector fields 
\begin{equation*}
    X_{p_i}=\frac{\partial}{\partial p_i}-\frac{1}{2}q_i\frac{\partial}{\partial t},\; X_{q_i}=\frac{\partial}{\partial q_i}+\frac{1}{2}p_i\frac{\partial}{\partial t},\; T= \frac{\partial}{\partial t},\quad i=1,2,\dots , n.
\end{equation*}
For each \(\lambda \in \mathbb{R}\setminus \{0\}\), the Heisenberg group admits an irreducible unitary representation \(\rho_{\lambda}\) on \(L^2(\mathbb{R}^{n})\), given by
 $$
 \left(\rho_\lambda(p,q,t)\varphi\right)(\xi)=e^{2\pi i\lambda t}e^{2\pi i \lambda (q \xi+\frac{1}{2}p q) }\varphi(\xi+p),
 $$
for \(\varphi \in L^2(\mathbb{R}^n)\). These representations are referred to as the Schrödinger representations of the Heisenberg group. By the Stone--von Neumann theorem, every irreducible unitary representation of \(\mathbb{H}^n\) that is nontrivial on the centre is unitary equivalence to a Schrödinger representation $\rho_\lambda$ for some $\lambda \neq 0$.

An alternative realization of the Heisenberg group, which is particularly useful in the present setting, is obtained through differential and multiplication operators acting on the Schwartz space \(S(\mathbb{R}^n)\); see \cite{rottensteiner-heisenberg_modulation_space}. More precisely, the canonical commutation relations satisfied by these operators generate the \((2n+1)\)-dimensional Heisenberg Lie algebra. This realization provides a natural connection between the representation theory of the Heisenberg group and the analysis of the differential operators and it will be useful in the construction of the Dynin--Folland group. 

\subsection{The Dynin--Folland group}

We now recall the Dynin--Folland group and its associated Lie algebra. Following \cite{rottensteiner-heisenberg_modulation_space}, consider the real Lie algebra
\[
\left\langle
2\pi i\,\mathcal{D}_{p_j},
2\pi i\,\mathcal{D}_{q_k},
2\pi i\,\mathcal{D}_{t},
2\pi i\,\mathcal{X}_{p_l},
2\pi i\,\mathcal{X}_{q_m},
2\pi i\,\mathcal{X}_{t}
\right\rangle
\]
generated by the left-invariant vector fields
\begin{equation}\label{first set of operators}
\left\{
\begin{aligned}
\mathcal{D}_{p_j}
&:= (2\pi i)^{-1}dR(X_{p_j})
= (2\pi i)^{-1}\left(\frac{\partial}{\partial p_j}
-\frac{1}{2}q_j\frac{\partial}{\partial t}\right),\\
\mathcal{D}_{q_k}
&:= (2\pi i)^{-1}dR(X_{q_k})
= (2\pi i)^{-1}\left(\frac{\partial}{\partial q_k}
+\frac{1}{2}p_k\frac{\partial}{\partial t}\right),\\
\mathcal{D}_{t}
&:= (2\pi i)^{-1}dR(X_t)
= (2\pi i)^{-1}\frac{\partial}{\partial t},
\end{aligned}
\right.
\end{equation}
together with the multiplications by coordinate functions
\begin{equation}\label{multiplication operators}
\left\{
\begin{aligned}
\mathcal{X}_{p_l}f(p,q,t)
&= p_lf(p,q,t),\\
\mathcal{X}_{q_m}f(p,q,t)
&= q_mf(p,q,t),\\
\mathcal{X}_{t}f(p,q,t)
&= tf(p,q,t),
\end{aligned}
\right.
\end{equation}
where $j,k,l,m=1,\ldots,n$ and $f\in\mathcal{S}(\mathbb{H}^n)$. The commutation relations between these differential and multiplication operators generate a 3-step nilpotent Lie algebra. More precisely, the real Lie algebra
\[
\left\langle
2\pi i\,\mathcal{D}_{p_j},
2\pi i\,\mathcal{D}_{q_k},
2\pi i\,\mathcal{D}_{t},
2\pi i\,\mathcal{X}_{p_l},
2\pi i\,\mathcal{X}_{q_m},
2\pi i\,\mathcal{X}_{t}
\right\rangle
\]
generated by the operators \eqref{first set of operators} and \eqref{multiplication operators} equals
\[
\begin{aligned}
&\mathbb{R}\,2\pi i\,\mathcal{D}_{p_1}\oplus\cdots\oplus
\mathbb{R}\,2\pi i\,\mathcal{D}_{p_n}\oplus
\mathbb{R}\,2\pi i\,\mathcal{D}_{q_1}\oplus\cdots\oplus
\mathbb{R}\,2\pi i\,\mathcal{D}_{q_n}\oplus
\mathbb{R}\,2\pi i\,\mathcal{D}_{t}\\
&\qquad\oplus\;
\mathbb{R}\,2\pi i\,\mathcal{X}_{p_1}\oplus\cdots\oplus
\mathbb{R}\,2\pi i\,\mathcal{X}_{p_n}\oplus
\mathbb{R}\,2\pi i\,\mathcal{X}_{q_1}\oplus\cdots\oplus
\mathbb{R}\,2\pi i\,\mathcal{X}_{q_n}\oplus
\mathbb{R}\,2\pi i\,\mathcal{X}_{t}\oplus
\mathbb{R}\,2\pi i\,I,
\end{aligned}
\]
where \(I\) denotes the identity operator on \(L^2(\mathbb{H}^n).\) This Lie algebra is $3$-step nilpotent and has topological dimension $2(2n+1)+1 = 4n+3$.

To describe the corresponding abstract Lie algebra, let
\[
(X_{u_1},\ldots,X_{u_n},
X_{v_1},\ldots,X_{v_n},
X_w,
X_{x_1},\ldots,X_{x_n},
X_{y_1},\ldots,X_{y_n},
X_z,
X_s)
\]
be the standard basis of $\mathbb{R}^{2(2n+1)+1}$. Consider the linear isomorphism
\begin{equation*}
d\pi:\mathbb{R}^{2(2n+1)+1}
\longrightarrow
\left\langle
2\pi i\,\mathcal{D}_{p_j},
2\pi i\,\mathcal{D}_{q_k},
2\pi i\,\mathcal{D}_{t},
2\pi i\,\mathcal{X}_{p_l},
2\pi i\,\mathcal{X}_{q_m},
2\pi i\,\mathcal{X}_{t}
\right\rangle
\end{equation*}
defined by
\[
\begin{aligned}
d\pi(X_{u_j})&=2\pi i\,\mathcal{D}_{p_j},&
d\pi(X_{v_j})&=2\pi i\,\mathcal{D}_{q_j},&
d\pi(X_w)&=2\pi i\,\mathcal{D}_{t},\\
d\pi(X_{x_j})&=2\pi i\,\mathcal{X}_{p_j},&
d\pi(X_{y_j})&=2\pi i\,\mathcal{X}_{q_j},\\
d\pi(X_z)&=2\pi i\,\mathcal{X}_{t},&
d\pi(X_s)&=2\pi i\,I.
\end{aligned}
\]

\begin{defi}
The Dynin--Folland Lie algebra $\mathfrak{h}_{n,2}$ is the real Lie algebra whose underlying linear space is $\mathbb{R}^{2(2n+1)+1}$ and whose Lie bracket $[\cdot,\cdot]_{\mathfrak{h}_{n,2}}$ is defined so that $d\pi$ is a Lie algebra morphism. 
\end{defi}
\noindent An important property of the Dynin--Folland Lie algebra $\mathfrak{h}_{n,2}$ is that it contains a natural copy of the Heisenberg Lie algebra $\mathfrak{h}_n$. More precisely, the subalgebra
\[
\langle 2\pi i\,\mathcal{D}_{p_j},\, 2\pi i\,\mathcal{D}_{q_k},\, 2\pi i\,\mathcal{D}_{t}\rangle
\]
is isomorphic to \(\mathfrak{h}_n\) and so is the subalgebra
\[
\mathbb{R}X_{u_1}\oplus\cdots\oplus\mathbb{R}X_{u_n}\oplus\mathbb{R}X_{v_1}
\oplus\cdots\oplus\mathbb{R}X_{v_n}\oplus\mathbb{R}X_{w}.
\]

The Dynin--Folland group \(\mathbb{H}_{n,2}\) is the connected, simply connected Lie group whose Lie algebra is isomorphic to $\mathfrak{h}_{n,2}$. Thus, after identifying the group with its Lie algebra via the exponential map, the group law can be expressed through the Baker--Campbell formula. More precisely, for $\mathbf{X}, \mathbf{X'} \in \mathbb{H}_{n,2}$, we have
\begin{equation*}
    \mathbf{X} \odot_{\mathbb{H}_{n,2}}\mathbf{X'} = \mathbf{Z},
\end{equation*}
\begin{equation*}
    \text{with} \quad \mathbf{Z} := \mathbf{X} + \mathbf{X'} + \frac{1}{2}[\mathbf{X}, \mathbf{X'}]_{\mathfrak{h}_{n,2}} + \frac{1}{12}[(\mathbf{X} - \mathbf{X'}), [\mathbf{X}, \mathbf{X'}]_{\mathfrak{h}_{n,2}}]_{\mathfrak{h}_{n,2}},
\end{equation*}
where \(\odot_{\mathbb{H}_{n,2}}\) denotes the group operation in the Dynin--Folland group.

\noindent Following the notation in \cite{rottensteiner-heisenberg_modulation_space}, we denote an element of the Heisenberg group by \[X=(p,q,t)\in \mathbb{R}^{2n+1},\] and an element of the Dynin--Folland group by \[\mathbf{X}=(P,Q,S)=\left((u,v,w),(x,y,z),s \right)\in \mathbb{R}^{2(2n+1)+1}.\] The inner product on \(\mathbb{R}^{2n+1}\) will be denoted by \(\langle \cdot, \cdot \rangle=\langle \cdot, \cdot \rangle_{\mathbb{R}^{2n+1}}.\)

Fischer et al.~\cite{rottensteiner-heisenberg_modulation_space} provided a complete classification of the unitary irreducible representations of the Dynin--Folland group. These representations fall into four distinct classes. Among them, the following family plays a fundamental role in the construction of modulation spaces on the Heisenberg group.

\noindent For every \(\lambda \in \mathbb{R}\setminus\{0\}\), there exists a unitary representation \(\pi_\lambda\) of the Dynin--Folland group \(\mathbb{H}_{n,2}\) on \(L^2(\mathbb{H}_n)\), given by
    \[
    (\pi_\lambda(P,Q,S)f)(X)
    =
    e^{2\pi i\lambda\left(S+\left\langle Q,\,X\cdot \left(\tfrac12 P\right)\right\rangle\right)}
    \,f(X\cdot P),
    \]
    for every \((P,Q,S)\in \mathbb{H}_{n,2}\), \(X\in \mathbb{H}_n\) and \(f\in L^2(\mathbb{H}_n)\), where $X \cdot P$ denotes the usual product in $\mathbb{H}^n$.

\noindent Moreover, if \(\lambda,\lambda'\in\mathbb{R}\setminus\{0\}\) with \(\lambda\neq\lambda'\), then the representations \(\pi_\lambda\) and \(\pi_{\lambda'}\) are unitarily inequivalent.

Apart from characterizing the unitarily equivalent representations of the Dynin--Folland group, Fischer et al.~ also described its projective unitary representations. In particular, they showed that the projective kernel associated with the representation \(\pi_{\lambda}\) coincides with the center $Z(\mathbb{H}_{n,2})$ of the Dynin--Folland group. Consequently, the projective representation corresponding to \(\pi_\lambda\) is given by
\[
\left(\pi^{\mathrm{pr}}_\lambda(P,Q)\psi\right)(X)
=
e^{2\pi i\lambda\left\langle Q,\,
X\cdot\left(\tfrac{1}{2}P\right)\right\rangle}
\,\psi(X\cdot P).
\]
Following \cite[Lemma 5.3]{rottensteiner-heisenberg_modulation_space}, the Dynin--Folland group \(\mathbb{H}_{n,2}\) can be realized as the semidirect product
\[
\mathbb{H}_{n,2}\cong \mathbb{R}^{2n+2}\rtimes \mathbb{H}_n.
\]
More precisely, every element of \(\mathbb{H}_{n,2}\) can be uniquely written in the form
\[
((Q,S),P):=(0,Q,S)\odot_{\mathbb{H}_{n,2}}(P,0,0),
\]
where \(P\in\mathbb{H}_n\), \(Q\in\mathbb{R}^{2n+1}\) and \(S\in\mathbb{R}\). In these coordinates, the representation \(\pi_\lambda\) takes the form
\[
\left(\pi_\lambda((Q,S),P)\psi\right)(X)
=
e^{2\pi i\lambda(S+\langle Q,X\rangle)}
\psi(X\cdot P).
\]
The Haar measure on \(\mathbb{H}_{n,2}\) in these coordinates is
\[
d\mu_{\mathbb{H}_{n,2}}((Q,S),P)=dS\,dQ\,dP.
\]
Furthermore, the associated projective representation is given by
\[
\bigl(\pi_\lambda^{\mathrm{pr}}(Q,P)\psi\bigr)(X)
=
e^{2\pi i\lambda\langle Q,X\rangle}
\psi(X\cdot P).
\]

Before introducing modulation spaces on the Heisenberg group, we recall the representation-theoretic interpretation of the classical modulation spaces on \(\mathbb{R}^n\). Fischer et al.~gave a detailed account of the realization of classical modulation spaces through the Schr\"odinger representation of the Heisenberg group. More precisely, they showed that the short-time Fourier transform can be interpreted as the transform associated with the Schr\"odinger representation, while the corresponding modulation space norm is obtained by equipping this transform with a suitable mixed Lebesgue norm. This representation-theoretic viewpoint provides a natural framework for extending the notion of modulation spaces beyond the Euclidean setting.

Following the same philosophy, Fisher et al.~\cite{rottensteiner-heisenberg_modulation_space} used a projective representation of the Dynin--Folland group to introduce a corresponding family of modulation spaces on the Heisenberg group, referred to as the Heisenberg-modulation spaces. We recall the precise definition below.

\begin{defi}
Let \(1\leq p,q\leq\infty\) and let \(g\in\mathcal{S}(\mathbb{R}^{2n+1})\). The nonweighted Heisenberg-modulation space \(E^{p,q}(\mathbb{R}^{2n+1})\) is defined by
\[
E^{p,q}(\mathbb{R}^{2n+1})
:=
\left\{
f\in\mathcal{S}'(\mathbb{R}^{2n+1})
:\,
\|f\|_{E^{p,q}(\mathbb{R}^{2n+1})}<\infty
\right\},
\]
where
\[
\|f\|_{E^{p,q}(\mathbb{R}^{2n+1})}
:=
\bigl\|
\langle f,\check{\pi}^{\mathrm{pr}}g\rangle
\bigr\|_{L^{p,q}\!\left(\mathbb{H}_{n,2}/Z(\mathbb{H}_{n,2})\right)}.
\]
\end{defi}
\noindent Here, \[\check{\pi}^{\mathrm{pr}}=\mathcal{F}^{-1} \pi^{\mathrm{pr}}\mathcal{F},\] and \[\bigl\|
\langle f,\check{\pi}^{\mathrm{pr}}g\rangle
\bigr\|_{L^{p,q}\!\left(\mathbb{H}_{n,2}/Z(\mathbb{H}_{n,2})\right)}\] denotes the corresponding mixed Lebesgue norm.

\noindent Thus, the quantity \[\langle f,\check{\pi}^{\mathrm{pr}}g\rangle\] can be regarded as an analogue of \textit{the short-time Fourier transform in the setting of the Heisenberg group}. To make this analogy more explicit, we introduce the translation and modulation operators on $L^2(\mathbb{H}^n)$ by
$$T_Pf(X) = f(X \cdot P), \qquad M_Qf(X) = e^{2\pi i \langle Q,X \rangle},$$
for $f \in L^2(\mathbb{H}^n)$ and $P,Q,X \in \mathbb{H}^n$. 

\noindent Using the relation
$$\check{\pi}^{\mathrm{pr}}(Q,P) = \mathcal{F}^{-1}\pi^{\mathrm{pr}}(Q,P)\mathcal{F},$$
the short-time Fourier transform on the Heisenberg group can be expressed as
    $$V_gf(Q,P) =\langle f,\check{\pi}^{\mathrm{pr}}(Q,P)g\rangle =\langle \mathcal{F}(f),\pi^{\mathrm{pr}}(Q,P)\mathcal{F}(g) \rangle.$$
    Consequently,
    $$V_gf(Q,P) = \int_{\mathbb{R}^{2n+1}}\mathcal{F}(f)(X) \overline{e^{2\pi i \langle Q,X \rangle} \mathcal{F}(g)(X \cdot P)}\,dX = \langle \mathcal{F}(f), M_Q T_P\mathcal{F}(g)\rangle.$$
    Equivalently, 
    $$V_gf(Q,P) = \int_{\mathbb{R}^{2n+1}}\mathcal{F}(f)(X) \overline{T_P\mathcal{F}(g)(X)} e^{-2\pi i \langle Q,X \rangle} \,dX = \mathcal{F}\left(\mathcal{F}(f)\,T_P\overline{\mathcal{F}(g)} \right)(Q).$$

The following proposition summarizes the fundamental properties of the short-time Fourier transform on the Heisenberg group.

\begin{prop}
The short-time Fourier transform satisfies the following properties.
\begin{enumerate}
    \item \textbf{Orthogonality relation.}
    For every $f_1,f_2,g_1,g_2\in L^2(\mathbb{H}^n)$,
    \begin{equation}\label{eq:orthogonality-relation}
        \bigl\langle V_{g_1}f_1,V_{g_2}f_2\bigr\rangle_
        {L^2(\mathbb{H}^n\times\mathbb{H}^n)}
        =
        \langle f_1,f_2\rangle_{L^2(\mathbb{H}^n)}
        \langle g_2,g_1\rangle_{L^2(\mathbb{H}^n)}.
    \end{equation}

    \item \textbf{Plancherel formula.}
    For every $f,g\in L^2(\mathbb{H}^n)$,
    \begin{equation}\label{eq:plancherel-formula}
        \|V_gf\|_{L^2(\mathbb{H}^n\times\mathbb{H}^n)}
        =
        \|f\|_{L^2(\mathbb{H}^n)}
        \|g\|_{L^2(\mathbb{H}^n)}.
    \end{equation}

    \item \textbf{Inversion formula.}
    Let $g,h\in L^2(\mathbb{H}^n)$ satisfy
    $\langle h,g\rangle_{L^2(\mathbb{H}^n)}\neq 0$. Then, for every
    $f\in L^2(\mathbb{H}^n)$,
    \begin{equation}\label{eq:inversion-formula}
        \mathcal{F}(f)
        =
        \frac{1}{\langle h,g\rangle_{L^2(\mathbb{H}^n)}}
        \int_{\mathbb{R}^{2n+1}}
        \int_{\mathbb{R}^{2n+1}}
        V_gf(Q,P)\,M_QT_P\mathcal{F}(h)\,dQ\,dP,
    \end{equation}
    where the integral is understood in the weak sense.
\end{enumerate}
\end{prop}
\begin{proof}
We denote
\[
F_j=\mathcal{F}(f_j)
\qquad\text{and}\qquad
G_j=\mathcal{F}(g_j),
\qquad j=1,2.
\]
Then for each $P\in\mathbb{H}^n$, we have
\[
V_{g_j}f_j(\cdot,P)
=
\mathcal{F}\bigl(F_j \, T_P\overline{G_j}\bigr).
\]
Hence, by the Plancherel theorem for the Euclidean Fourier transform,
\begin{align*}
\bigl\langle V_{g_1}f_1,V_{g_2}f_2\bigr\rangle
&=
\int_{\mathbb{H}^n}
\left\langle
V_{g_1}f_1(\cdot,P),
V_{g_2}f_2(\cdot,P)
\right\rangle\,dP                                          \\
&=
\int_{\mathbb{H}^n}
\left\langle
\mathcal{F}\bigl(F_1 \,T_P\overline{G_1}\bigr),
\mathcal{F}\bigl(F_2 \, T_P\overline{G_2}\bigr)
\right\rangle\,dP                                          \\
&=
\int_{\mathbb{H}^n}
\left\langle
F_1 \, T_P\overline{G_1},
F_2 \, T_P\overline{G_2}
\right\rangle\,dP                                          \\
&=
\int_{\mathbb{R}^{2n+1}}
\int_{\mathbb{R}^{2n+1}}
F_1(Q)\overline{F_2(Q)}
\overline{G_1(Q\cdot P)}G_2(Q\cdot P)
\,dQ\,dP.
\end{align*}
Since the Haar measure on $\mathbb{H}^n$ is right invariant, the
change of variables $R=Q\cdot P$, followed by Fubini's theorem,
gives
\begin{align*}
\bigl\langle V_{g_1}f_1,V_{g_2}f_2\bigr\rangle
&=
\left(
\int_{\mathbb{R}^{2n+1}}
F_1(Q)\overline{F_2(Q)}\,dQ
\right)
\left(
\int_{\mathbb{R}^{2n+1}}
\overline{G_1(R)}G_2(R)\,dR
\right)                                                     \\
&=
\langle F_1,F_2\rangle\langle G_2,G_1\rangle                \\
&=
\langle f_1,f_2\rangle\langle g_2,g_1\rangle.
\end{align*}
This proves the orthogonality relation.
    \noindent \textit{(2)} follows immediately from \textit{(1)} by taking
$f_1=f_2=f$ and $g_1=g_2=g$.

\noindent
To prove the inversion formula, define
\[
\widetilde{F}
:=
\frac{1}{\langle h,g\rangle}
\int_{\mathbb{R}^{2n+1}}
\int_{\mathbb{R}^{2n+1}}
V_gf(Q,P)\,M_QT_P\mathcal{F}(h)\,dQ\,dP,
\]
where the integral is understood in the weak sense. For any
$\gamma\in L^2(\mathbb{H}^n)$, the orthogonality relation yields
\begin{align*}
\left\langle \widetilde{F},\mathcal{F}(\gamma)\right\rangle
&=
\frac{1}{\langle h,g\rangle}
\int_{\mathbb{R}^{2n+1}}
\int_{\mathbb{R}^{2n+1}}
V_gf(Q,P)\,
\overline{
\left\langle
\mathcal{F}(\gamma),M_QT_P\mathcal{F}(h)
\right\rangle}
\,dQ\,dP                                                   \\
&=
\frac{1}{\langle h,g\rangle}
\int_{\mathbb{R}^{2n+1}}
\int_{\mathbb{R}^{2n+1}}
V_gf(Q,P)\,\overline{V_h\gamma(Q,P)}\,dQ\,dP               \\
&=
\frac{1}{\langle h,g\rangle}
\left\langle V_gf,V_h\gamma\right\rangle                   \\
&=
\frac{1}{\langle h,g\rangle}
\langle f,\gamma\rangle\langle h,g\rangle                  \\
&=
\langle f,\gamma\rangle
=
\left\langle\mathcal{F}(f),\mathcal{F}(\gamma)\right\rangle.
\end{align*}
Since $\gamma\in L^2(\mathbb{H}^n)$ is arbitrary and $\mathcal{F}$ is
unitary, it follows that
\[
\widetilde{F}=\mathcal{F}(f).
\]
This proves the inversion formula.
\end{proof}

As in the case of the short-time Fourier transform on \(\mathbb{R}^n\), an application of the Cauchy--Schwarz inequality yields 
\[|V_gf(Q,P)|\leq \lVert \mathcal{F}(f)\rVert_2 \lVert M_QT_P \mathcal{F}(g) \rVert_2 =\lVert \mathcal{F}(f)\rVert_2 \lVert \mathcal{F}(g)\rVert_2=\lVert f\rVert_2 \lVert g\rVert_2.  \]
Consequently, 
\begin{equation}\label{inf norm of Vgf}
\lVert V_gf \rVert_{\infty}\leq \lVert f\rVert_2 \lVert g\rVert_2. 
\end{equation}

For a nonzero window function \(g \in L^2(\mathbb{H}^n)\), the following lemma shows that the image of \(L^2(\mathbb{H}^n)\) under the transform \(V_g\) forms a reproducing kernel Hilbert space.

\begin{lemma}\label{reproducing kernel lemma}
Let $g\in L^2(\mathbb{H}^n)$ be nonzero. Then
$V_g(L^2(\mathbb{H}^n))$ is a reproducing kernel Hilbert space with
reproducing kernel
\begin{equation} \label{reproducing kernel} 
\begin{aligned}
K_g\bigl((Q',P'),(Q,P)\bigr)
&=
\frac{1}{\|g\|_2^2}
V_g\bigl(\check{\pi}^{\mathrm{pr}}(Q,P)g\bigr)(Q',P') \\
&=
\frac{1}{\|g\|_2^2}
V_g\!\left(
\mathcal{F}^{-1}
\bigl(M_QT_P\mathcal{F}(g)\bigr)
\right)(Q',P').
\end{aligned}
\end{equation}
Moreover,
\[
\left|K_g\bigl((Q',P'),(Q,P)\bigr)\right|\leq 1.
\]
\end{lemma}

\begin{proof}
By the Plancherel formula,
\[
\|V_gf\|_2=\|g\|_2\|f\|_2,
\qquad f\in L^2(\mathbb{H}^n).
\]
Since $g\neq0$, the operator $\|g\|_2^{-1}V_g$ is an isometry.
Consequently, $V_g(L^2(\mathbb{H}^n))$ is a closed subspace of
$L^2(\mathbb{H}^n\times\mathbb{H}^n)$.

For every $f\in L^2(\mathbb{H}^n)$, the orthogonality relation gives
\begin{align*}
V_gf(Q,P)
&=
\left\langle
f,\check{\pi}^{\mathrm{pr}}(Q,P)g
\right\rangle                                             \\
&=
\frac{1}{\|g\|_2^2}
\left\langle
V_gf,
V_g\bigl(\check{\pi}^{\mathrm{pr}}(Q,P)g\bigr)
\right\rangle                                             \\
&=
\left\langle
V_gf,
K_g\bigl((\cdot,\cdot),(Q,P)\bigr)
\right\rangle.
\end{align*}
Thus, point evaluation is continuous on
$V_g(L^2(\mathbb{H}^n))$, and \eqref{reproducing kernel} is its
reproducing kernel.

Finally, using
\[
|V_gu(Q',P')|\leq \|u\|_2\|g\|_2
\]
and the unitarity of $\check{\pi}^{\mathrm{pr}}(Q,P)$, we obtain
\begin{align*}
\left|K_g\bigl((Q',P'),(Q,P)\bigr)\right|
&\leq
\frac{1}{\|g\|_2^2}
\left\|\check{\pi}^{\mathrm{pr}}(Q,P)g\right\|_2
\|g\|_2                                                   \\
&=1.
\end{align*}
This completes the proof.
\end{proof}


\section{Benedicks' Theorem}\label{sec Benedicks}

In this section, we establish an analogue of Benedicks' theorem for the short-time Fourier transform on the Heisenberg group. Recall that the classical Benedicks theorem asserts that a nonzero function and its Euclidean Fourier transform cannot both be supported on sets of finite Lebesgue measure. In the present setting, for each fixed $P\in\mathbb{H}^n$, the transform can be written as
\[
V_gf(Q,P)
=
\mathcal{F}\left(
\mathcal{F}(f)\,T_P\overline{\mathcal{F}(g)}
\right)(Q).
\]
This representation allows us to apply the classical Benedicks theorem to the function
\[
\mathcal{F}(f)\,T_P\overline{\mathcal{F}(g)}
\]
for almost every $P$. The finite-support assumption on $\mathcal{F}(g)$ controls the support of this function, while the finite-support assumption on $V_gf$ controls the support of its Euclidean Fourier transform. Combining these observations with the invariance of Haar measure under Heisenberg translations yields the following desired qualitative uncertainty principle.

\begin{theorem}\label{benedicks}
    Suppose that $f,g \in L^2(\mathbb{H}^n)$ satisfy $$|\operatorname{supp}(V_gf)| < \infty \quad \text{and} \quad |\operatorname{supp}(\mathcal{F}(g))|<\infty.$$ Then either \(f=0\) or \(g=0\).
\end{theorem}

    \begin{proof}
       Assume that \(f,g \in L^2(\mathbb{H}^n)\) satisfy $$|\operatorname{supp}(V_gf)| < \infty \quad \text{and} \quad |\operatorname{supp}(\mathcal{F}(g))|<\infty.$$
       If \(g=0\), then there is nothing to prove. Hence, we may assume that $g \neq 0$. Let $G = \mathcal{F}(g)$.
       
       \noindent For every $Q,P \in \mathbb{H}^n$, we have
       \begin{equation}
           V_gf(Q,P) = \langle \mathcal{F}(f), \pi^{pr}(Q,P) \mathcal{F}(g) \rangle =\mathcal{F}\left(\mathcal{F}(f)\,T_P\overline{G}\right)(Q). \label{vgf in terms of h_p} 
       \end{equation}
    Since
\(
V_gf\in L^2(\mathbb H^n\times\mathbb H^n),
\) it follows that
\[
\int_{\mathbb{R}^{2n+1}}|V_gf(Q,P)|^2\,dQ<\infty,
\qquad\text{for a.e. }P.
\]
Using \eqref{vgf in terms of h_p}, we obtain 
\[
\int_{\mathbb{R}^{2n+1}}
\left|
\mathcal{F}\left(\mathcal{F}(f)\,T_P\overline{G}\right)(Q)
\right|^2
\,dQ<\infty,
\qquad\text{for a.e. }P.
\]
An application of the Plancherel theorem therefore yields
\[
\int_{\mathbb R^{2n+1}}
\left|
(\mathcal{F}(f)\,T_P\overline{G})(X)
\right|^2
\,dX<\infty,
\qquad\text{for a.e. }P,
\]
and hence
\[
\mathcal{F}(f)\,T_P\overline{G}\in
L^2(\mathbb R^{2n+1}),
\qquad\text{for a.e. }P.
\]
Moreover,
    $$\operatorname{supp}\left(\mathcal{F}(f)\, T_P\overline{G}\right) \subseteq \operatorname{supp} \left(\mathcal{F}(f)\right) \cap \operatorname{supp} \left(T_P\overline{G}\right) \subseteq \operatorname{supp} \left(T_P\overline{G}\right).$$
     Since \(|\operatorname{supp}(G)| = |\operatorname{supp}(\mathcal{F}(g))|<\infty\), it follows that \(|\operatorname{supp}(T_P\overline{G})|<\infty\) for every (and hence for almost every) $P$. Consequently, $$|\operatorname{supp}(\mathcal{F}(f)\, T_P\overline{G})|< \infty, \qquad \text{for a.e. } P.$$
Furthermore, by \eqref{vgf in terms of h_p},
$$\mathcal{F}\left(\mathcal{F}(f)\,T_P\overline{G}\right) = V_gf(\cdot,P).$$
Since $|\operatorname{supp}(V_gf)| < \infty$, we deduce that
$$|\operatorname{supp}(\mathcal{F}(\mathcal{F}(f)\,T_P\overline{G}))|<\infty, \qquad \text{for a.e. } P.$$
By applying Benedicks' theorem to the function \(\mathcal{F}(f)\,T_P\overline{G} \in L^2(\mathbb{R}^{2n+1})\), it follows that $$\mathcal{F}(f)\,T_P\overline{G}=0, \qquad \text{for almost every } P \in \mathbb{H}^n.$$
Integrating with respect to $X$ and $P$ and applying Fubini's theorem, we obtain
     \begin{align*}
        0 &=\int_{\mathbb{R}^{2n+1}} \int_{\mathbb{R}^{2n+1}} |(\mathcal{F}(f)\,T_P\overline{G})(X)|^2\,dX\,dP\\
        &= \int_{\mathbb{R}^{2n+1}} |\mathcal{F}(f)(X)|^2 \left(\int_{\mathbb{R}^{2n+1}} |\overline{G(X \cdot P)}|^2\,dP \right)\, dX\\
        &= \int_{\mathbb{R}^{2n+1}} |\mathcal{F}(f)(X)|^2 \left(\int_{\mathbb{R}^{2n+1}} |G(P)|^2\,dP \right)\, dX\\
        &=\lVert \mathcal{F}(f) \rVert_2^2 \, \lVert G \rVert_2^2.
     \end{align*}
     Since \(G = \mathcal{F}(g) \neq 0\), we have $\|G\|_2 > 0$. Therefore, $\| \mathcal{F}(f)\|_2 = 0$ and hence $f= 0$ by the Plancherel theorem. This completes the proof.
    \end{proof}


\section{Donoho--Stark Uncertainty Principle}\label{sec:donoho-stark}

In contrast to the qualitative support theorem established in the preceding section, the Donoho--Stark uncertainty principle provides a quantitative restriction on the concentration of the short-time Fourier transform. More precisely, it gives a lower bound for the measure of a phase-space region on which a prescribed proportion of the $L^2$-energy of $V_gf$ is concentrated. Our approach uses the reproducing kernel structure of the range of $V_g$ and the interaction between the corresponding range projection and a phase-space localization operator.

Let $P_g$ denote the orthogonal projection from $L^2(\mathbb{H}^n\times\mathbb{H}^n)$ onto the closed subspace $V_g(L^2(\mathbb{H}^n))$. For a measurable subset $\Omega\subseteq\mathbb{H}^n\times\mathbb{H}^n$, let $P_\Omega$ denote the orthogonal projection defined by
\[
P_\Omega F:=\chi_\Omega F,
\qquad
F\in L^2(\mathbb{H}^n\times\mathbb{H}^n).
\]
The key observation is that, when $\Omega$ has finite measure, the localized projection $P_\Omega P_g$ is a Hilbert--Schmidt operator whose norm is determined by $|\Omega|$. This operator estimate leads
directly to lower bounds for the portion of $V_gf$ lying outside $\Omega$ and subsequently, to the Donoho--Stark concentration inequalities.

\begin{prop}
    Let \(g\in L^2(\mathbb{H}^n)\) be nonzero. Then, for every measurable subset \(\Omega \subseteq \mathbb{H}^n \times \mathbb{H}^n\) of finite measure, the operator \(P_\Omega P_g\) is a Hilbert--Schmidt operator satisfying
    \[\lVert P_\Omega P_g \rVert_{HS}= \sqrt{|\Omega}|. \]
\end{prop}

\begin{proof}
     By Lemma \ref{reproducing kernel lemma}, the space $V_g(L^2(\mathbb{H}^n))$ is a reproducing kernel Hilbert space with reproducing kernel $K_g$. Consequently, the orthogonal projection $P_g$ is the integral operator with kernel $K_g$. Hence, the operator \(P_\Omega P_g\) is an integral operator whose kernel is 
    \begin{equation*}
        \mathcal{N}((Q',P'),(Q,P))=\chi_\Omega(Q,P) K_g((Q',P'),(Q,P)),
    \end{equation*}
    where \(K_g\) is given by \eqref{reproducing kernel}.

    \noindent Therefore,
    \begin{align*}
        \lVert P_\Omega P_g \rVert^2_{HS}&=\lVert \mathcal{N}\rVert^2_2\\
        &=\int_{\mathbb{H}^n \times \mathbb{H}^n} \int_{\mathbb{H}^n \times \mathbb{H}^n} |\chi_{\Omega}(Q,P)|^2 |K_g((Q',P'),(Q,P))|^2\,dQ'\,dP'\,dQ\,dP\\
        &=\frac{1}{\lVert g \rVert_2^4}\int_{\mathbb{H}^n \times \mathbb{H}^n} |\chi_{\Omega}(Q,P)|^2 \left( \int_{\mathbb{H}^n \times \mathbb{H}^n} \lvert V_g(\mathcal{F}^{-1}(M_QT_P \mathcal{F}(g)))(Q',P')\rvert^2\,dQ'\,dP' \right)\,dQ\,dP\\
        &=\frac{1}{\lVert g \rVert_2^4}\int_{\mathbb{H}^n \times \mathbb{H}^n} |\chi_{\Omega}(Q,P)|^2 \lVert g \rVert_2^2 \lVert M_QT_P\mathcal{F}(g)\rVert_2^2\,dQ\,dP\\
        &=\int_{\mathbb{H}^n \times \mathbb{H}^n} |\chi_{\Omega}(Q,P)|^2\,dQ\,dP =\lvert \Omega \rvert,
    \end{align*}
    which completes the proof.
\end{proof}

In the following proposition, we prove that the short-time Fourier transform cannot be concentrated on a measurable set of sufficiently small measure.

\begin{prop}\label{stft concentration on set of small measure}
    Let \(g \in L^2(\mathbb{H}^n)\) be a nonzero function and let \(\Omega \subseteq \mathbb{H}^n \times \mathbb{H}^n\) be a measurable set satisfying \(\lvert \Omega \rvert<1.\) Then, for every \(f \in L^2(\mathbb{H}^n)\),
    \[\lVert \chi_{\Omega^c}V_gf \rVert_2\geq \sqrt{1-\lvert \Omega \rvert} \, \lVert f \rVert_2\lVert g\rVert_2.\]
\end{prop}
\begin{proof}
    For any \(f\in L^2(\mathbb{H}^n)\), we have 
    \begin{align*}
\lVert V_gf \rVert_2^2
&=
\left\lVert
\chi_{\Omega}V_gf+\chi_{\Omega^c}V_gf
\right\rVert_{2}^2
\\
&=
\left\lVert
\chi_{\Omega}V_gf
\right\rVert_{2}^2
+
\left\lVert
\chi_{\Omega^c}V_gf
\right\rVert_{2}^2
\\
&\le
\lvert \Omega \rvert\,
\lVert V_gf \rVert_{\infty}^2
+
\left\lVert
\chi_{\Omega^c}V_gf
\right\rVert_{2}^2
\\
&\le
\lvert \Omega \rvert\,
\lVert f\rVert_{2}^2
\lVert g\rVert_{2}^2
+
\left\lVert
\chi_{\Omega^c}V_gf
\right\rVert_{2}^2.
\end{align*}
    Invoking the Plancherel formula \eqref{eq:plancherel-formula}, we obtain
    \[\lVert \chi_{\Omega^c}V_gf \rVert_2^2 \geq (1-\lvert \Omega \rvert) \, \lVert f \rVert_2^2 \lVert g\rVert_2^2,\]
    and taking the square roots yields the desired inequality.
\end{proof}

Let \(\Omega\) be a measurable subset of \(\mathbb{H}^n \times \mathbb{H}^n\) and \(f,g\) in \(L^2(\mathbb{H}^n)\). Assume that \(0 < \epsilon_\Omega < 1\). We say that \(V_gf\) is \(\epsilon_{\Omega}\)-concentrated on \(\Omega\) if 
 \begin{equation}\label{epsilon omega concentrated}
    \|\chi_{\Omega^{c}} V_gf\|_{2} \le \epsilon_{\Omega} \|V_gf\|_{2}.
    \end{equation}

We now establish the Donoho--Stark uncertainty principle for the short-time Fourier transform on the Heisenberg group.

\begin{theorem}
    Let \(f,g \in L^2(\mathbb{H}^n)\) be nonzero functions. Suppose that \(\Omega \subseteq \mathbb{H}^n \times \mathbb{H}^n\) is measurable and that \(0 < \epsilon_\Omega < 1\). If \(V_gf\) is \(\epsilon_{\Omega}\)-concentrated on \(\Omega\), then 
    \[\lvert \Omega \rvert \geq (1-\epsilon_{\Omega}^2).\]
\end{theorem}
\begin{proof}
    By the decomposition,
    $$V_gf = \chi_\Omega V_gf + \chi_{\Omega^c} V_gf$$
    and using \eqref{epsilon omega concentrated}, we obtain
    \begin{align*}
        \lVert V_gf \rVert_2^2&=\lVert \chi_{\Omega}V_gf \rVert_2^2+\lVert \chi_{\Omega^c}V_gf \rVert_2^2\\
        &\leq \lVert \chi_{\Omega}V_gf \rVert_2^2+\epsilon^2_{\Omega}\lVert V_gf \rVert_2^2.
    \end{align*}
    Hence, 
    \begin{equation}\label{boundindonoho}
        (1-\epsilon^2_{\Omega})\lVert V_gf \rVert_2^2 \leq \|\chi_\Omega V_gf\|_2^2.
    \end{equation}
    Using Plancherel formula \eqref{eq:plancherel-formula}, we obtain 
    \begin{align*}
        (1-\epsilon^2_{\Omega})\lVert f \rVert_2^2 \lVert g \rVert_2^2 &\leq \lVert \chi_{\Omega}V_gf \rVert_2^2\\
        &\le \lvert \Omega \rvert \, \lVert \chi_{\Omega}V_gf \rVert^2_{\infty}\\
        &\le \lvert \Omega \rvert \, \lVert f \rVert_2^2 \lVert g \rVert_2^2.
    \end{align*}
    Since $f$ and $g$ are nonzero, it follows that \[\lvert \Omega \rvert \geq (1-\epsilon_{\Omega}^2),\]
    thereby completing the proof.
\end{proof}

Recall that 
$$\|V_gf\|_2 = \|f\|_2 \|g\|_2 \quad \text{and} \quad \|V_gf\|_\infty \leq \|f\|_2 \|g\|_2.$$
Therefore, by the Riesz--Thorin interpolation theorem, it follows that
\begin{equation}\label{rieszthorin}
\|V_gf\|_p \leq \|f\|_2 \|g\|_2, \qquad \text{for every } p \geq 2.
\end{equation}
By applying H\"{o}lder's inequality with conjugate exponents $p/2$ and $p/(p-2)$ to \eqref{boundindonoho} and then using \eqref{rieszthorin} for $p > 2$, we obtain the following result.

\begin{corollary}
    Let \(f,g \in L^2(\mathbb{H}^n)\) be nonzero functions. Suppose that \(\Omega \subseteq \mathbb{H}^n \times \mathbb{H}^n\) is measurable and that \(0 < \epsilon_\Omega < 1\). If \(V_gf\) is \(\epsilon_{\Omega}\)-concentrated on \(\Omega\), then for every \(p>2\), 
    \[\lvert \Omega \rvert \geq (1-\epsilon_{\Omega}^2)^{\frac{p}{p-2}}.\]
\end{corollary}


\section{Lieb's Uncertainty Principle}\label{sec Lieb}

In this section, we establish Lieb's inequality, a fundamental result that provides \(L^p\)-norm estimates for the short-time Fourier transform. The proof relies on the sharp Hausdorff--Young inequality on $\mathbb{R}^{2n+1}$ together with the sharp Young convolution inequality on $\mathbb{H}^n$, thereby highlighting the close interplay between Fourier analysis and convolution estimates.
 
 \begin{theorem}
     Let \(f,g \in L^2(\mathbb{H}^n)\) and \(p>2\). Then
\begin{equation}\label{resultlieb}
    \int_{\mathbb{H}^n} \int_{\mathbb{H}^n}\left|V_gf(Q,P)\right|^p dQ\, dP
\leq
\left(\frac{2}{p}\right)^{2n+1}
\left(\|f\|_2 \, \|g\|_2\right)^p.
\end{equation}
 \end{theorem}

\begin{proof}
    Let \(p'\) denote the H\"older conjugate of \(p\). Since \(f,g \in L^2(\mathbb{H}^n)\), the Plancherel theorem implies that \(\mathcal{F}(f), \mathcal{F}(g) \in L^2(\mathbb{H}^n)\). Hence, by H\"older's inequality, 
    $$\mathcal{F}(f) \, T_P \overline{\mathcal{F}(g)} \in L^1(\mathbb{H}^n)$$
    for every $P \in \mathbb{H}^n$.
    
    \noindent Moreover, 
    $$V_gf(Q,P)=\mathcal{F}\left(\mathcal{F}(f) \, T_P \overline{\mathcal{F}(g)}\right)(Q)\in L^2(\mathbb{H}^n \times \mathbb{H}^n).$$ 
    Therefore, by Fubini's theorem,
    $$\mathcal{F}(f) \, T_P \overline{\mathcal{F}(g)} \in L^1(\mathbb{H}^n) \cap L^2(\mathbb{H}^n)$$
    for almost every \(P \in \mathbb{H}^n\). Since $1 < p' < 2$, the interpolation theorem for Lebesgue spaces implies that
    $$\mathcal{F}(f) \, T_P \overline{\mathcal{F}(g)} \in L^{p'} (\mathbb{H}^n)$$
    for almost every \(P \in \mathbb{H}^n\). Applying the sharp Hausdorff--Young inequality \cite[(1.4)]{Foundations-of-time-frequency-analysis}, we obtain, for almost every fixed $P \in \mathbb{H}^n$,
    \begin{align*}
        \left( \int_{\mathbb{H}^n}\left|V_gf(Q,P)\right|^p\,dQ\right)^{1/p}&=\left( \int_{\mathbb{R}^{2n+1}}\left|V_gf(Q,P)\right|^p dQ\right)^{1/p}\\
        &= \left( \int_{\mathbb{R}^{2n+1}}\left|\mathcal{F}\left( \mathcal{F}(f) T_P \overline{ \mathcal{F}(g)} \right)(Q)\right|^p dQ\right)^{1/p}\\
        &\leq A^{2n+1}_{p'}\left( \int_{\mathbb{R}^{2n+1}}\left| \mathcal{F}(f)(Q) \, \overline{\mathcal{F}(g)(Q \cdot P)}\right|^{p'} dQ\right)^{1/{p'}},
    \end{align*}
    where 
    $$A_{p'} = \left(\frac{(p')^{1/{p'}}}{p^{1/p}}\right)^{1/2}.$$
    Now define
    $$F(Q) = \mathcal{F}(f)(Q) \quad \text{and} \quad G(Q) = \overline{\mathcal{F}(g)(Q^{-1})}.$$
    Then
    \begin{align*}
        \left( \int_{\mathbb{H}^n}\left|V_gf(Q,P)\right|^p\,dQ\right)^{1/p}
        &\leq A^{2n+1}_{p'}\left( \int_{\mathbb{R}^{2n+1}} \left| F(Q) \right|^{p'} \left| G(P^{-1}\cdot Q^{-1} )\right|^{p'} dQ\right)^{1/{p'}}\\
        &= A^{2n+1}_{p'} \left( \left|G\right|^{p'} \ast \left|F\right|^{p'} (P^{-1})\right)^{1/{p'}},
    \end{align*}
    where $\ast$ denotes convolution on the Heisenberg group, defined for suitable functions \(\phi,\;\psi\) on \(\mathbb{H}^n\) by
    \[(\phi \ast \psi)(P)= \int_{\mathbb{H}^n}\phi(P \cdot Q^{-1}) \psi(Q)\,dQ.\]
    Taking the $L^p$-norm with respect to $P$ gives 
    \begin{equation}\label{equation for p norm of Vgf}
        \lVert V_gf \rVert_p \leq A^{2n+1}_{p'} \lVert |G|^{p'} \ast |F|^{p'} \rVert^{1/p'}_{p/p'}.
    \end{equation}
    The next step is to invoke the sharp Young convolution inequality on the Heisenberg group (see \cite{Mchrist-young-convolution}). Applying this inequality with the admissible exponent triple \[\left( \frac{2}{p'},\frac{2}{p'},\frac{p}{p'} \right),\] we obtain 
\[\lVert |G|^{p'} \ast |F|^{p'}\rVert_{p/p'}\leq \left(A^2_{2/p'} A_{p/p'}\right)^{2n+1} \lVert |G|^{p'}\rVert_{2/p'}\, \lVert |F|^{p'}\rVert_{2/p'},\]
where
\[
A_{2/p'}
=
\left[
\left(\frac{2}{p'}\right)^{\frac{p'}{2}}
\left(\frac{2}{2-p'}\right)^{-\frac{2-p'}{2}}
\right]^{\frac12},
\]
and 
\[
A_{p/p'}
=
\left[
\left(\frac{p}{p'}\right)^{\frac{p'}{p}}
\left(\frac{p-1}{p-2}\right)^{-\frac{p-2}{p-1}}
\right]^{\frac12}.
\]
Using the above estimate in \eqref{equation for p norm of Vgf} gives 
\begin{align*}
    \lVert V_gf \rVert_{p} &\leq A^{2n+1}_{p'}\left(\left(A^2_{2/p'} A_{p/p'}\right)^{2n+1}\lVert |G|^{p'}\rVert_{2/p'}\, \lVert |F|^{p'}\rVert_{2/p'}\right)^{1/p'}\\
    &\leq A^{2n+1}_{p'} \left(A^2_{2/p'} A_{p/p'}\right)^{\frac{2n+1}{p'}} \left(\lVert G \rVert^{p'}_{2}\,\lVert F \rVert^{p'}_{2}\right)^{1/p'}\\
     &\leq A^{2n+1}_{p'}\left(A^2_{2/p'} A_{p/p'}\right)^{\frac{2n+1}{p'}}\lVert \mathcal{F}(g)\rVert_{2}\,\lVert \mathcal{F}(f) \rVert_{2}.
\end{align*}
 Substituting the explicit values of the constants and applying the Plancherel theorem yields the desired form of Lieb's inequality.
\end{proof}

\begin{remark}
   It is worth noting that the above proof relies fundamentally on the sharp forms of the Hausdorff--Young inequality on $\mathbb{R}^{2n+1}$ and the Young convolution inequality on $\mathbb{H}^n$. In the Euclidean setting, Gaussian functions are extremizers for both inequalities and hence equality is attained in Lieb's inequality, showing that the inequality is indeed sharp. In contrast, Young's convolution inequality on the Heisenberg group admits no extremizers (see \cite{Mchrist-young-convolution}). Consequently, the above argument does not allow us to conclude that \eqref{resultlieb} is sharp. Moreover, in the Euclidean setting, the proof for the case \(1\leq p\leq2\) relies on the converse of Young's inequality, which to the best of our knowledge is not yet known for the Heisenberg group.
\end{remark}

The following corollary is an immediate consequence of Lieb's inequality and H\"older's identity.

\begin{corollary}
    Let \(f, g \in L^2(\mathbb{H}^n)\). Suppose that \(\Omega \subseteq \mathbb{H}^n \times \mathbb{H}^n\) and \(\epsilon>0\) satisfy 
    \[\iint_{\Omega}\lvert V_gf(Q,P) \rvert^2\,dQ\,dP \geq (1-\epsilon) \lVert f \rVert_2 \lVert g \rVert_2.\]
    Then \[\lvert \Omega \rvert \geq (1-\epsilon)^{\frac{p}{p-2}}\left( \frac{p}{2} \right)^{\frac{2(2n+1)}{p-2}}, \qquad \text{for every } p>2.\]
\end{corollary}


\section{Hirschman Uncertainty Principle}\label{sec Hirschman}

Let $f,g \in L^2(\mathbb{H}^n)$, where the window $g$ is normalized so that $\|g\|_2 = 1$. If $f$ is also normalized, i.e., $\|f\|_2 = 1$, then the Plancherel formula \eqref{eq:plancherel-formula} implies that the spectrogram $$\mathcal{V}(Q,P) := \lvert V_gf(Q,P) \rvert^2$$ is a probability density function on $\mathbb{H}^n \times \mathbb{H}^n$. The Shannon entropy of the spectrogram is defined by
$$E(\mathcal{V}) := - \int_{\mathbb{H}^n} \int_{\mathbb{H}^n} \mathcal{V}(Q,P) \ln \mathcal{V}(Q,P) \, dQ dP.$$

\begin{remark}
    We emphasize that the normalization assumptions on $f$ and $g$ are imposed solely to ensure that $\lvert V_gf \rvert^2$ is a probability density function. More generally, for arbitrary $f,g \in L^2(\mathbb{H}^n)$, the quantity $E(\lvert V_gf \rvert^2)$ remains well defined, although it is no longer the Shannon entropy of a probability density function.
\end{remark}

We are now in a position to establish the Hirschman uncertainty principle for the short-time Fourier transform on the Heisenberg group.

\begin{theorem}
    Let $f,g \in L^2(\mathbb{H}^n)$. Then
    \begin{equation}\label{resultentropy}
        E(\lvert V_gf \rvert^2) \geq \left( (2n+1) - 2 \ln(\|f\|_2 \|g\|_2) \right) \|f\|_2^2 \|g\|_2^2.
    \end{equation}
\end{theorem}

\begin{proof}
    Let $f,g \in L^2(\mathbb{H}^n)$ be such that $\|f\|_2 = \|g\|_2 = 1$. If $E(\mathcal{V}) = E(|V_gf|^2) = \infty$, then the desired inequality follows immediately. Therefore, it suffices to assume that $E(\mathcal{V}) < \infty$.
    
    \noindent By Lieb's inequality \eqref{resultlieb}, for every $p > 2$, 
    $$\int_{\mathbb{H}^n} \int_{\mathbb{H}^n} |V_gf(Q,P)|^p dQ dP \leq \left( \frac{2}{p} \right)^{2n+1}.$$
    Setting $\beta = p/2$, we have $\beta > 1$ and consequently,
    $$ \int_{\mathbb{H}^n} \int_{\mathbb{H}^n} \mathcal{V}(Q,P)^\beta dQ dP \leq \beta^{-(2n+1)}.$$
    Taking the natural logarithm of both sides and dividing by $1 - \beta$, we obtain
    $$\frac{1}{1 - \beta} \, \ln \left( \int_{\mathbb{H}^n} \int_{\mathbb{H}^n} \mathcal{V}(Q,P)^\beta dQ dP \right) \geq -(2n+1) \frac{\ln \beta}{1 - \beta}.$$
    Passing to the limit as $\beta \to 1^{+}$, it follows that
    $$\lim_{\beta \to 1^{+}} \frac{1}{1 - \beta} \, \ln \left( \int_{\mathbb{H}^n} \int_{\mathbb{H}^n} \mathcal{V}(Q,P)^\beta dQ dP \right) \geq (2n+1).$$
    Applying L'H\^{o}pital's rule to the left-hand side yields
    $$-\frac{\displaystyle\lim_{\beta \to 1^{+}} \frac{d}{d\beta}\left(\int_{\mathbb{H}^n}\int_{\mathbb{H}^n} |\mathcal{V}(Q,P)|^\beta\,dQ\,dP \right)} {\displaystyle\lim_{\beta \to 1^{+}} \int_{\mathbb{H}^n}\int_{\mathbb{H}^n} \mathcal{V}(Q,P)^\beta \,dQ\,dP}.$$
    Since $E(\mathcal{V}) < \infty$, differentiation under the integral sign is justified by the dominated convergence theorem. Consequently,
    $$-\frac{\displaystyle\lim_{\beta \to 1^{+}} \int_{\mathbb{H}^n}\int_{\mathbb{H}^n} \mathcal{V}(Q,P)^\beta \, \ln \mathcal{V}(Q,P) \,dQ\,dP} {\displaystyle\lim_{\beta \to 1^{+}} \int_{\mathbb{H}^n}\int_{\mathbb{H}^n} \mathcal{V}(Q,P)^\beta \,dQ\,dP}.$$
    A further application of the dominated convergence theorem permits the interchange of the limit and the integral and therefore
    $$- \int_{\mathbb{H}^n}\int_{\mathbb{H}^n} \mathcal{V}(Q,P) \, \ln \mathcal{V}(Q,P) \,dQ\,dP \geq (2n+1).$$
    Since $\mathcal{V} = |V_gf|^2$, it follows that
    $$E(|V_gf|^2) \geq (2n+1).$$

    Now let $f,g \in L^2(\mathbb{H}^n)$ be arbitrary and define 
    $$\phi = \frac{f}{\|f\|_2} \qquad \text{and} \qquad \psi = \frac{g}{\|g\|_2}.$$
    The preceding argument implies that
    \begin{equation}\label{entropy}
    E(|V_\psi\phi|^2) \geq (2n+1).    
    \end{equation}
    Moreover, a direct computation shows that
    $$V_\psi\phi(Q,P) = \frac{1}{\|f\|_2 \|g\|_2} V_gf(Q,P).$$
    Hence,
    \begin{align*}
    E(|V_\psi\phi|^2) &= - \int_{\mathbb{H}^n} \int_{\mathbb{H}^n} |V_\psi\phi(Q,P)|^2 \ln |V_\psi\phi(Q,P)|^2 \, dQ dP\\
    &= - \int_{\mathbb{H}^n} \int_{\mathbb{H}^n} \frac{1}{\|f\|_2^2 \|g\|_2^2} |V_gf(Q,P)|^2 \ln \left( \frac{1}{\|f\|_2^2 \|g\|_2^2} |V_gf(Q,P)|^2 \right) dQ dP\\
    &= - \frac{1}{\|f\|_2^2 \|g\|_2^2} \int_{\mathbb{H}^n} \int_{\mathbb{H}^n} |V_gf(Q,P)|^2 \ln |V_gf(Q,P)|^2 \, dQ dP + 2 \ln (\|f\|_2 \|g\|_2).
    \end{align*}
    Using \eqref{entropy}, we obtain
    $$\frac{1}{\|f\|_2^2 \|g\|_2^2} E(|V_gf|^2) + 2 \ln (\|f\|_2 \|g\|_2) \geq (2n+1)$$
    and consequently,
    $$E(\lvert V_gf \rvert^2) \geq \left( (2n+1) - 2 \ln(\|f\|_2 \|g\|_2) \right) \|f\|_2^2 \|g\|_2^2,$$
    which is precisely the desired inequality \eqref{resultentropy}.
\end{proof}

\begin{remark}
    In 2024, Poria and the second author established an analogous entropy uncertainty principle for the short-time Fourier transform on the lattice $\mathbb{Z}^n\times\mathbb{T}^n$; see~\cite[Theorem~3.20]{poriadasgupta2024}. The constant in the lower bound obtained here is stronger than the corresponding constant appearing in the lattice setting.
\end{remark}


\section{Beurling--Hardy-type Uncertainty Principle}
\label{sec BeurlingHardy}

In this section, we establish a qualitative uncertainty principle that combines features of the classical Beurling and Hardy theorems. We assume that the short-time Fourier transform $V_gf$ satisfies a Beurling-type weighted integrability condition, while the Euclidean Fourier transform $\mathcal{F}(f)$ exhibits Gaussian decay of Hardy-type. These two assumptions impose excessive simultaneous decay in the phase-space and Fourier variables. Throughout the remainder of the paper, we denote the Euclidean norm of $Q \in \mathbb{H}^n \cong \mathbb{R}^{2n+1}$ by $|Q|$.

The proof again exploits the representation
\[
V_gf(Q,P)
=
\mathcal{F}\left(
\mathcal{F}(f)\,T_P\overline{\mathcal{F}(g)}
\right)(Q).
\]
For almost every fixed $P\in\mathbb{H}^n$, the imposed decay conditions yield a classical Beurling-type condition for
\[
\mathcal{F}(f)\,T_P\overline{\mathcal{F}(g)}.
\]
Applying the Euclidean Beurling theorem fiberwise and then integrating with respect to the Heisenberg translation parameter leads to the conclusion that either $f=0$ or $g=0$. The precise statement of thet result is as follows.

\begin{theorem}
    Suppose that 
    \begin{equation}\label{beurlingtypecondition}
    \int_{\mathbb{H}^n} \int_{\mathbb{H}^n} \left\lvert V_gf(Q,P) \right\rvert e^{\pi (|Q|^2 + |P|^2)} dQ dP < \infty,    
    \end{equation}
    and 
    \begin{equation}\label{hardytypecondition}
        \mathcal{F}(f)(X) = O(e^{-a |X|^2}),
    \end{equation}
    for some $a > \pi$. Then either $f = 0$ or $g = 0$.
\end{theorem}

\begin{proof}
    If $g = 0$, then there is nothing to prove. Hence, we may assume that $g \neq 0$. 

    For each fixed $P \in \mathbb{H}^n$, consider the integral
    $$I := \int_{\mathbb{R}^{2n+1}} \int_{\mathbb{R}^{2n+1}} \left\lvert (\mathcal{F}(f) T_P \overline{\mathcal{F}(g)})(X) \; \mathcal{F}\left(\mathcal{F}(f) T_P \overline{\mathcal{F}(g)}\right)(Q) \right\rvert e^{2 \pi |\langle X,Q \rangle|} dX dQ.$$
    Then
    \begin{align*}
        I &\leq \int_{\mathbb{R}^{2n+1}} \int_{\mathbb{R}^{2n+1}} \left\lvert (\mathcal{F}(f) T_P \overline{\mathcal{F}(g)})(X) \; \mathcal{F}\left(\mathcal{F}(f) T_P \overline{\mathcal{F}(g)}\right)(Q) \right\rvert e^{\pi (|X|^2 + |Q|^2)} dX dQ\\
        &= \left( \int_{\mathbb{H}^{n}} \left\lvert (\mathcal{F}(f) T_P \overline{\mathcal{F}(g)})(X) \right\rvert e^{\pi |X|^2} dX \right) \left( \int_{\mathbb{H}^{n}} \left\lvert \mathcal{F}(\mathcal{F}(f) T_P \overline{\mathcal{F}(g)})(Q) \right\rvert e^{\pi |Q|^2} dQ \right).
    \end{align*}
    Since $$V_gf(Q,P) = \mathcal{F}\left(\mathcal{F}(f)\,T_P\overline{\mathcal{F}(g)} \right)(Q),$$ it follows immediately from \eqref{beurlingtypecondition} that
    $$\int_{\mathbb{H}^{n}} \left\lvert \mathcal{F}(\mathcal{F}(f) T_P \overline{\mathcal{F}(g)})(Q) \right\rvert e^{\pi |Q|^2} dQ < \infty$$
    for almost every $P \in \mathbb{H}^n$.
    
    \noindent Next, by H\"{o}lder inequality together with \eqref{hardytypecondition}, we have
    \begin{align*}
        \int_{\mathbb{H}^{n}} \left\lvert (\mathcal{F}(f) T_P \overline{\mathcal{F}(g)})(X) \right\rvert e^{\pi |X|^2} dX &\leq \left( \int_{\mathbb{H}^{n}} |\mathcal{F}(f)(X)|^2 e^{2 \pi |X|^2} dX \right)^{1/2} \left( \int_{\mathbb{H}^{n}} |T_P \overline{\mathcal{F}(g)}(X)|^2 dX \right)^{1/2}\\
        &\lesssim \left(\int_{\mathbb{H}^{n}} e^{-2 (a - \pi) |X|^2} dX \right)^{1/2} \|g\|_2 < \infty,
    \end{align*}
    since $a > \pi$. This implies that the integral $I$ is finite for almost every $P \in \mathbb{H}^n$.
    
    \noindent Therefore, Beurling's uncertainty principle on $\mathbb{R}^{2n+1}$ yields that 
    $$\mathcal{F}(f) \, T_P \overline{\mathcal{F}(g)} = 0, \qquad \text{for a.e. } P.$$
    Finally, arguing exactly as in the proof of the Benedicks theorem (Theorem \ref{benedicks}), we conclude that $f = 0$, which completes the proof.
\end{proof}


\section{Heisenberg-type Inequality and Local Uncertainty Principles}\label{sec Heisenberg}

In this section, we establish the Heisenberg-type uncertainty inequality, which provides a quantitative lower bound for the weighted $L^2$-norms of the short-time Fourier transform with respect to both variables, thereby showing that these weighted norms cannot be simultaneously made arbitrarily small. As an application, we derive a local uncertainty estimate that controls the concentration of the short-time Fourier transform on measurable subsets of finite measure. Finally, we establish local uncertainty principles in the spirit of Price for the short-time Fourier transform on the Heisenberg group.

\begin{theorem}
  For every \(s>0\), there exists a constant \(C_s>0\) such that, for every \(f,g \in L^2(\mathbb{H}^n)\),
  \[
  \lVert \, \lvert Q \rvert^s \, V_gf \rVert^2_{2}+ \lVert \, \lvert P \rvert^s \, V_gf \rVert^2_{2} \geq C_s \lVert f \rVert^2_2 \lVert g \rVert^2_2.
  \]
\end{theorem}

\begin{proof}
    For any \(\delta>0\), let  
    \[B_\delta := \{(Q,P) \in \mathbb{H}^n \times \mathbb{H}^n: \lvert (Q,P)\rvert<\delta\},\]
    where $|(Q,P)|$ denotes the Euclidean norm of $(Q,P) \in \mathbb{H}^n \times \mathbb{H}^n$. Choose \(\delta>0\) sufficiently small so that \(\lvert B_\delta \rvert < 1\). By Proposition \ref{stft concentration on set of small measure}, it follows that
    \[\iint_{B^c_\delta} \lvert V_gf(Q,P) \rvert^2\,dQ\,dP \geq \left(1-\lvert B_\delta \rvert \right)\lVert f \rVert^2_2 \lVert g \rVert^2_2.\]
Hence,
\begin{align}\label{weighted norm Vgf}
    \lVert f \rVert^2_2 \lVert g \rVert^2_2 &\leq \frac{1}{1-\lvert B_\delta \rvert}\iint_{B^c_\delta} \lvert V_gf(Q,P) \rvert^2\,dQ\,dP \nonumber \\
    &\leq \frac{1}{\delta^{2s}\left(1-\lvert B_\delta \rvert\right)}\iint_{\mathbb{H}^n \times \mathbb{H}^n
    } \lvert (Q,P) \rvert^{2s} \, \lvert V_gf(Q,P) \rvert^2\,dQ\,dP\nonumber\\
    &\leq \frac{1}{\delta^{2s}\left(1-\lvert B_\delta \rvert\right)} \left\lVert \, \lvert (Q,P) \rvert^{s} \, V_gf \right\rVert_2^2.
\end{align}
    Finally, using the elementary inequality 
    $$|(Q,P)|^{2s} \leq 2^s (|Q|^{2s} + |P|^{2s}),$$
we obtain
    $$\lVert f \rVert^2_2 \lVert g \rVert^2_2 \leq \frac{2^s}{\delta^{2s}\left(1-\lvert B_\delta \rvert\right)} \left(\lVert \, \lvert Q \rvert^s \, V_gf \rVert^2_2+  \lVert \, \lvert P \rvert^s \, V_gf \rVert^2_2\right).$$
Setting \[C_s=\frac{\delta^{2s}\left(1-\lvert B_\delta \rvert\right)}{2^s}\] gives the desired inequality.
\end{proof}

The preceding theorem immediately yields a local uncertainty estimate.

\begin{theorem}
     For every \(s>0\), there exists a constant \( \tilde{C}_s > 0 \) such that, for every \(f,g \in L^2(\mathbb{H}^n)\) and every measurable set \(\Omega \subseteq \mathbb{H}^n \times \mathbb{H}^n\) of finite measure,
  \begin{equation*}
  \lVert \chi_{\Omega} \, V_gf \rVert_{2} \leq \tilde{C}_s \, \lvert \Omega \rvert^{1/2} \left\lVert \, \lvert (Q,P)\rvert^s \, V_gf\right\rVert_2.      
  \end{equation*}
\end{theorem}

\begin{proof}
    Using the estimate \[\lVert V_gf \rVert_{\infty} \leq \|f\|_2 \|g\|_2,\] together with \eqref{weighted norm Vgf}, we obtain
    \begin{equation}\label{localinequality}
        \lVert \chi_{\Omega} V_gf \rVert_2 \leq \lvert \Omega \rvert^{1/2}\|V_gf\|_\infty \leq \lvert \Omega \rvert^{1/2} \lVert f \rVert_2 \lVert g \rVert_2 \leq \frac{\lvert \Omega \rvert^{1/2}}{\delta^{s}\sqrt{ 1-\lvert B_\delta \rvert}} \lVert \, \lvert (Q,P) \lvert^s \, V_gf \rVert_2, 
    \end{equation}
    where \(\delta > 0\) is chosen so that \(\lvert B_\delta \rvert<1.\)
\end{proof}

As an immediate consequence, we obtain the following lower bound for the weighted $L^2$-norm of the short-time Fourier transform.

\begin{corollary}
    For every \(s>0\), there exists a constant \(\mathfrak{c}_s>0\) such that, for every \(f,g \in L^2(\mathbb{H}^n)\)
    \[\lVert \, \lvert (Q,P) \rvert^s \, V_gf \rVert_2 \geq \mathfrak{c}_s \lVert f \rVert_2 \lVert g \rVert_2.\]
\end{corollary}

\begin{proof}
    For any \(\delta>0\), let  
    \[B_\delta := \{(Q,P) \in \mathbb{H}^n \times \mathbb{H}^n: \lvert (Q,P) \rvert < \delta\}.\]
    Using the Plancherel formula \eqref{eq:plancherel-formula} and inequality \eqref{localinequality}, we obtain
    \begin{align*}
    \|f\|_2^2 \|g\|_2^2 = \|V_gf\|_2^2 &= \|\chi_{B_\delta} V_gf\|_2^2 + \|\chi_{B_\delta^c} V_gf\|_2^2\\    &\leq \frac{|B_\delta|}{\delta^{2s} (1-|B_\delta|)} \lVert \, \lvert (Q,P) \rvert^s \, V_gf \rVert_2^2 + \delta^{-2s} \lVert \, \lvert (Q,P) \lvert^s \, V_gf \rVert_2^2\\
    &= \frac{1}{\delta^{2s} (1 - |B_\delta|)} \lVert \, \lvert (Q,P) \rvert^s \, V_gf \rVert_2^2,
    \end{align*}
    for every $\delta > 0$ satisfying $|B_\delta| < 1$.
    
    \noindent Rearranging the previous inequality yields
    $$\lVert \, \lvert(Q,P)\rvert^s \, V_gf \rVert_2^2 \geq \delta^{2s} (1 - |B_\delta|) \lVert f \rVert_2^2 \lVert g \rVert_2^2,$$
    for every $\delta > 0$ satisfying $|B_\delta| < 1$. Taking the supremum over all such $\delta$, we conclude that
    $$\lVert \, \lvert(Q,P)\rvert^s \, V_gf \rVert_2 \geq \mathfrak{c}_s \lVert f \rVert_2 \lVert g \rVert_2,$$
    where
    $$0 < \mathfrak{c}_s = \left( \sup_{\substack{\delta>0 \\ |B_\delta| < 1}} \delta^{2s} (1-|B_\delta|) \right)^{1/2} < \infty.$$
    This completes the proof.
\end{proof}

The preceding theorem controls the concentration of the short-time Fourier transform on measurable sets of finite measure in terms of its weighted $L^2$-norm. We now establish Price's local uncertainty principles, in which the concentration of the short-time Fourier transform is instead controlled by weighted norm of the Fourier transform of the underlying function. This result may be viewed as an analogue of Price's local uncertainty principle for the Euclidean Fourier transform in the setting of the short-time Fourier transform on the Heisenberg group.

In 1986, Price proved the following local uncertainty principle for the Euclidean Fourier transform \cite[Theorem 1.1]{Price-sharp_local_uncertainty}. If \(E \subseteq \mathbb{R}^n\) is measurable and \(\alpha > n/2\), then
\begin{equation*}
   \int_{E}| \widehat{f}(\xi)|^2 d\xi \lesssim |E| \; \lVert f \rVert_2^{2-n/\alpha} \lVert \, |x|^{\alpha}f \rVert_2^{n/\alpha},
\end{equation*}
for every \(f \in L^2(\mathbb{R}^n)\), where \(\widehat{f}\) denotes the Fourier transform of \(f\) in \(\mathbb{R}^n\).

Motivated by Price's result, we establish an analogous local uncertainty principle for the short-time Fourier transform on the Heisenberg group. Recall that the Heisenberg group \(\mathbb{H}^{n}\) has homogeneous dimension \(\mathcal{Q} = 2n + 2\) and the \emph{Korányi norm} on \(\mathbb{H}^{n}\) is given by  
\[
    |(z,t)| = \big(|z|^{4} + t^{2}\big)^{1/4},
\]
which is homogeneous of degree \(1\) with respect to the natural dilations of \(\mathbb{H}^{n}\) (see \cite{Pritam-Lacunary_spherical_function}). Henceforth, for $Q = (z,t) \in \mathbb{H}^n$, we write $\|Q\|$ to denote its Kor\'{a}nyi norm.

The precise statement of the corresponding local uncertainty principle is as follows.

\begin{theorem}
    Let $g \in L^2(\mathbb{H}^n)$ satisfy $\mathcal{F}(g) \in L^\infty(\mathbb{H}^n)$. If \(\Omega \subseteq \mathbb{H}^n \times \mathbb{H}^n\) is measurable and \(\alpha>\mathcal{Q}/2\), then
    \[ \lVert \chi_{\Omega}V_g f \rVert_2^2 \lesssim \lvert \Omega \rvert \, \|f\|_2^{2-\mathcal{Q}/\alpha} \,\left\| \, \|Q\|^\alpha \mathcal{F}(f)\right\|_2^{\mathcal{Q}/ \alpha},\]
    for every $f \in L^2(\mathbb{H}^n)$.
\end{theorem}

\begin{proof}
    If 
\[
\lvert \Omega \rvert \, \|f\|_2^{2-\mathcal{Q}/\alpha} \,\left\| \, \|Q\|^\alpha \mathcal{F}(f)\right\|_2^{\mathcal{Q}/ \alpha} = \infty,
\]
then there is nothing to prove. Hence, we may assume that the above quantity is finite.

\noindent Since $|\Omega|<\infty$, we have
\begin{align*}
\|\chi_{\Omega}V_gf\|_2^2
&\leq
|\Omega|
\sup_{(Q,P)\in\mathbb{H}^n\times\mathbb{H}^n}
|V_gf(Q,P)|^2 \\
&=
|\Omega|
\left(
\sup_{(Q,P)\in\mathbb{H}^n\times\mathbb{H}^n}
\left|
\mathcal{F}\left(
\mathcal{F}(f)\,T_P\overline{\mathcal{F}(g)}
\right)(Q)
\right|
\right)^2 \\
&\leq
|\Omega|
\left(
\sup_{P\in\mathbb{H}^n}
\left\lVert
\mathcal{F}(f)\,T_P\overline{\mathcal{F}(g)}
\right\rVert_1
\right)^2.
\end{align*}

\noindent By H\"older's inequality, 
\[
\left\lVert
\mathcal{F}(f)\,T_P\overline{\mathcal{F}(g)}
\right\rVert_1
\leq
\lVert\mathcal{F}(f)\rVert_1
\lVert T_P\overline{\mathcal{F}(g)}\rVert_\infty.
\]
Since $\mathcal{F}(g) \in L^\infty(\mathbb{H}^n)$, we have
\[
\sup_{P\in\mathbb{H}^n}
\lVert T_P\overline{\mathcal{F}(g)}\rVert_\infty
=
\lVert\mathcal{F}(g)\rVert_\infty.
\]
Consequently,
\begin{equation}\label{estimate of Vgf over omega 2 norm}
\|\chi_{\Omega}V_gf\|_2^2
\leq
|\Omega|\,
\lVert\mathcal{F}(g)\rVert_\infty^2
\lVert\mathcal{F}(f)\rVert_1^2.
\end{equation}
    Next, by applying \cite[Lemma 5.2]{dabra2026uncertainty} with $p = 1$ and $q = 2$, we obtain
    \begin{equation}\label{estimate from blms paper}
        \lVert \mathcal{F}(f)\rVert_1 \lesssim \Vert \mathcal{F}(f) \rVert_2^{1 - \mathcal{Q}/2\alpha} \, \| \, \|Q\|^\alpha \, \mathcal{F}(f) \|_2^{\mathcal{Q}/2\alpha},
    \end{equation}
       where \(\alpha> \mathcal{Q}/2.\) Finally, combining the  estimates \eqref{estimate of Vgf over omega 2 norm} and \eqref{estimate from blms paper} yields the desired inequality.
\end{proof}

We now consider the complementary range \(0 < \alpha < \mathcal{Q}/2\). Our approach is motivated by the local uncertainty principle established in \cite{Price-Sitaram-Local_uncertainty}. The following lemma plays a crucial role in proving the corresponding local uncertainty principle in our framework. Since its proof follows verbatim from \cite[Theorem 1.1]{Price-Sitaram-Local_uncertainty}, we omit the details.

\begin{lemma}\label{pricelemma}
    Let \(f \in L^2(\mathbb{H}^n)\) and \(E \subseteq \mathbb{H}^n\) be measurable. Then
    \begin{equation*}
        \lVert \chi_E \mathcal{F}(f) \rVert_2^2 \lesssim \lvert E \rvert^{2\alpha/\mathcal{Q}} \; \lVert \, \| \cdot \|^\alpha f \rVert_2^2.
    \end{equation*}
\end{lemma}

As a consequence of the preceding lemma, we obtain the following Price-type local uncertainty principle for the short-time Fourier transform on the Heisenberg group.

\begin{theorem}
Let $g\in L^2(\mathbb{H}^n)$ satisfy $\|g\|_2=1$, and let
$\Omega\subseteq\mathbb{H}^n\times\mathbb{H}^n$ be measurable. If
$0<\alpha<\mathcal{Q}/2$, then for every
$f\in L^2(\mathbb{H}^n)$,
\[
\|\chi_\Omega V_gf\|_2^2
\lesssim
|\Omega|\,
\bigl\|\|\cdot\|^\alpha f\bigr\|_2 \,
\bigl\|\|\cdot\|^\alpha\mathcal{F}(f)\bigr\|_2.
\]
The implicit constant depends only on $n$ and $\alpha$.
\end{theorem}

\begin{proof}
Set
\[
A:=\bigl\|\|\cdot\|^\alpha f\bigr\|_2
\qquad\text{and}\qquad
B:=\bigl\|\|\cdot\|^\alpha\mathcal{F}(f)\bigr\|_2.
\]
If $f=0$, $|\Omega|=0$, or $|\Omega|AB=\infty$, then the desired
estimate is immediate. We may therefore assume that
\[
f\neq0,\qquad 0<|\Omega|<\infty,
\qquad\text{and}\qquad 0<A,B<\infty.
\]
Since $\|g\|_2=1$, the pointwise estimate for the short-time Fourier
transform gives
\[
\|\chi_\Omega V_gf\|_2^2
\leq
|\Omega|\,\|V_gf\|_\infty^2
\leq
|\Omega|\,\|f\|_2^2.
\]

For $r>0$, let
\[
B_r:=\bigl\{P\in\mathbb{H}^n:\|P\|<r\bigr\}.
\]
Using the Plancherel formula, Lemma~\ref{pricelemma}, and the estimate
$\|Q\|\geq r$ on $B_r^c$, we obtain
\begin{align*}
\|f\|_2^2
&=
\|\mathcal{F}(f)\|_2^2                                      \\
&=
\|\chi_{B_r}\mathcal{F}(f)\|_2^2
+
\|\chi_{B_r^c}\mathcal{F}(f)\|_2^2                         \\
&\lesssim
|B_r|^{2\alpha/\mathcal{Q}}A^2
+
r^{-2\alpha}B^2.
\end{align*}
Since the Korányi balls satisfy
\[
|B_r|=|B_1|r^{\mathcal{Q}},
\]
it follows that
\[
\|f\|_2^2
\lesssim
r^{2\alpha}A^2+r^{-2\alpha}B^2.
\]
Choosing
\[
r=\left(\frac{B}{A}\right)^{1/(2\alpha)}
\]
balances the two terms on the right-hand side and yields
\[
\|f\|_2^2\lesssim AB.
\]
Consequently,
\[
\|\chi_\Omega V_gf\|_2^2
\lesssim
|\Omega|\,
\bigl\|\|\cdot\|^\alpha f\bigr\|_2
\bigl\|\|\cdot\|^\alpha\mathcal{F}(f)\bigr\|_2,
\]
which completes the proof.
\end{proof}


\section{Decay Properties}\label{sec:decay}

In this section, we investigate the decay of the short-time Fourier transform $V_gf$ on the Heisenberg group in terms of ultradifferentiable regularity properties of
\[
F=\mathcal{F}(f)
\qquad\text{and}\qquad
G=\mathcal{F}(g).
\]
More precisely, we show that suitable quantitative bounds on the Euclidean derivatives of $F$ and $G$ yield estimates for arbitrary polynomial weights in the modulation variable $Q$. In this way,
regularity of the Fourier transforms of the signal and the window is converted into decay of the associated time-frequency representation.

The argument is based on the identity
\[
Q^\alpha V_gf(Q,P)
=
\frac{1}{(2\pi i)^{|\alpha|}}
\mathcal{F}\left(
\partial^\alpha
\bigl(F\,T_P\overline{G}\bigr)
\right)(Q),
\]
followed by the multi-index Leibniz rule. A feature specific to the present setting is that the Heisenberg translation $T_P$ does not commute with ordinary Euclidean differentiation. Differentiating
$T_PG$ produces additional terms whose coefficients are polynomial functions of the translation parameter $P$. Consequently, the resulting decay estimates are uniform when $P$ ranges over a bounded subset of $\mathbb{H}^n$.

To control the derivative terms arising from the Leibniz expansion, we use the sequence
\[
M_0^{\tau,\sigma}:=1,
\qquad
M_k^{\tau,\sigma}:=k^{\tau k^\sigma},
\qquad k\geq1,
\]
where $\tau>0$ and $\sigma>1$. This sequence satisfies the submultiplicative-type estimate
\[
M_{k-\ell}^{\tau,\sigma}M_\ell^{\tau,\sigma}
\leq
M_k^{\tau,\sigma},
\qquad 0\leq\ell\leq k.
\]
This property allows the derivative bounds for $F$ and $G$ to be combined efficiently after applying the Leibniz rule. Sequences of this form arise naturally in extended Gevrey regularity and more generally, in the theory of ultradifferentiable function spaces. For more details, see \cite{gevrey}.

We use the standard multi-index notation. For $$Q=(Q_1,Q_2,\dots,Q_{2n+1})\in\mathbb{R}^{2n+1}$$ and
$$\alpha=(\alpha_1,\alpha_2,\dots,\alpha_{2n+1})\in\mathbb{N}^{2n+1},$$ we write
\[
Q^\alpha=Q_1^{\alpha_1}Q_2^{\alpha_2}\cdots Q_{2n+1}^{\alpha_{2n+1}},
\qquad
|\alpha|=\alpha_1+\alpha_2+\cdots+\alpha_{2n+1},
\]
and
\[
\partial^\alpha
=\partial_{Q_1}^{\alpha_1}\partial_{Q_2}^{\alpha_2}\cdots
\partial_{Q_{2n+1}}^{\alpha_{2n+1}}.
\]

\noindent The following theorem gives the precise decay estimate. The formulation of this result is motivated by \cite[Theorem 4]{gevrey}.

\begin{theorem}
Let \(\tau>0\) and \(\sigma > 1\). Suppose that \(g \in L^2(\mathbb{H}^n)\) satisfies 
\begin{equation}\label{condition on G}
    \lVert \partial^{\gamma} G \rVert_1 \leq C_g^{|\gamma|^\sigma}|\gamma|^{\tau |\gamma|^\sigma}, \quad \text{for all } \gamma \in \mathbb{N}^{2n+1},
\end{equation}
and that \(f\in L^2(\mathbb{H}^n)\) satisfies
\begin{equation}\label{condition on F}
     \lVert \partial^{\gamma} F \rVert_\infty \leq C_f^{|\gamma|^\sigma}|\gamma|^{\tau |\gamma|^\sigma}, \quad \text{for all }\gamma \in \mathbb{N}^{2n+1}.
\end{equation}
Then, for every $\alpha \in \mathbb{N}^{2n+1}$, there exists a constant $C_{f,g} > 0$ such that, for \(P\) in a bounded subset of \(\mathbb{H}^n\),
\begin{equation*}
    \lvert Q^{\alpha}V_gf(Q,P) \rvert \leq \frac{1}{\pi^{|\alpha|}}
C_{f,g}^{|\alpha|^\sigma}|\alpha|^{\tau|\alpha|^{\sigma}} \left( |\alpha|+1 \right)^{2n}.
\end{equation*}
\end{theorem}

\begin{proof}
By the standard differentiation property of the Euclidean Fourier transform, for every
multi-index $\alpha$, we have
$$Q^{\alpha}V_gf(Q,P) = \frac{1}{(2\pi i)^{|\alpha|}}
\mathcal{F}\left(\partial^{\alpha}\left(F \, T_P\overline{G}\right) \right)(Q).$$
Applying the multi-index Leibniz rule gives
$$Q^{\alpha}V_gf(Q,P) = \frac{1}{(2\pi i)^{|\alpha|}} \sum_{\beta\leq\alpha} \binom{\alpha}{\beta} \mathcal{F}\left(\partial^{\alpha-\beta}F \, \partial^\beta(T_P\overline{G}) \right)(Q).$$
Hence, by the estimate $\|\mathcal{F}(h)\|_\infty \leq \|h\|_1$ and H\"older's inequality, we obtain
\begin{align}
\left|Q^{\alpha}V_gf(Q,P)\right|
&\leq
\frac{1}{(2\pi)^{|\alpha|}}
\sum_{\beta\leq\alpha}
\binom{\alpha}{\beta}
\left\lVert
\mathcal{F}\left(
\partial^{\alpha-\beta}F\,
\partial^\beta(T_P\overline{G})
\right)
\right\rVert_\infty \nonumber\\
&\leq
\frac{1}{(2\pi)^{|\alpha|}}
\sum_{\beta\leq\alpha}
\binom{\alpha}{\beta}
\left\lVert
\partial^{\alpha-\beta}F\,
\partial^\beta(T_P\overline{G})
\right\rVert_1 \nonumber\\
&\leq
\frac{1}{(2\pi)^{|\alpha|}}
\sum_{\beta\leq\alpha}
\binom{\alpha}{\beta}
\left\lVert
\partial^{\alpha-\beta}F
\right\rVert_\infty
\left\lVert
\partial^\beta(T_P\overline{G})
\right\rVert_1.
\label{bound for mod of lhs}
\end{align}

We next estimate the $L^1$ norm of $\partial^\beta(T_P\overline{G})$. The translation operator $T_P$ on the Heisenberg group is not compatible with ordinary Euclidean derivatives. In contrast to the Euclidean setting, where derivatives commute with translations, the noncommutative structure of the Heisenberg group gives rise to additional terms when differentiating a translated function. More precisely, by the chain rule, for every multi-index $\beta$, we have
\begin{equation*}
\partial^\beta T_P\overline{G}
=
T_P\left(
\sum_{|\delta|=|\beta|}
C_{\beta,\delta}(P)\,
\partial^\delta\overline{G}
\right),
\end{equation*}
where $C_{\beta,\delta}(P)$ are polynomial functions of the coordinates of $P$. Since $P$ ranges over a bounded subset of $\mathbb{H}^n$, there exists a constant $C\geq1$, such that
\begin{equation*}
\left|C_{\beta,\delta}(P)\right|
\leq C^{|\beta|},
\quad |\delta|=|\beta|.
\end{equation*}
Using the assumed estimate \eqref{condition on G} for the derivatives of $G$ with $\gamma = \delta$, we obtain
\begin{align*}
    \left\lVert \partial^{\beta}(T_P \overline{G}) \right \rVert_1& \leq C^{|\beta|} \sum_{|\delta|=|\beta|} \left\lVert \partial^{\delta}G \right\rVert_1\\
    &\leq C^{|\beta|} \sum_{|\delta|=|\beta|} C_g^{|\delta|^\sigma} |\delta|^{\tau|\delta|^\sigma}\\
    &=C^{|\beta|} C_g^{|\beta|^\sigma} |\beta|^{\tau|\beta|^\sigma} \binom{|\beta|+2n}{2n}.
\end{align*}
Substituting the estimate \eqref{condition on F} with $\gamma = \alpha - \beta$, together with the above estimate into \eqref{bound for mod of lhs}, we have 
\begin{align*}
\left| Q^{\alpha}V_gf(Q,P)\right|&\leq \frac{1}{(2\pi)^{|\alpha|}}
\sum_{\beta\leq\alpha}
\binom{\alpha}{\beta} C_f^{|\alpha-\beta|^\sigma}|\alpha-\beta|^{\tau|\alpha-\beta|^{\sigma}}C^{|\beta|}C_g^{|\beta|^\sigma}|\beta|^{\tau|\beta|^\sigma} \binom{|\beta|+2n}{2n}\\
& \leq \left(\frac{C}{2\pi}\right)^{|\alpha|} \binom{|\alpha|+2n}{2n} \sum_{\beta\leq\alpha}
\binom{\alpha}{\beta} C_f^{|\alpha-\beta|^\sigma}|\alpha-\beta|^{\tau|\alpha-\beta|^{\sigma}} C_g^{|\beta|^\sigma}|\beta|^{\tau|\beta|^\sigma}\\
& \leq \frac{1}{(2\pi)^{|\alpha|}} C_{f,g}^{|\alpha|^\sigma}|\alpha|^{\tau|\alpha|^{\sigma}} \binom{|\alpha|+2n}{2n} \sum_{\beta\leq\alpha}
\binom{\alpha}{\beta}\\
& \leq \frac{1}{\pi^{|\alpha|}}
C_{f,g}^{|\alpha|^\sigma}|\alpha|^{\tau|\alpha|^{\sigma}} \left( |\alpha|+1 \right)^{2n}.
\end{align*}
This proves the desired estimate.
\end{proof}


\section*{Acknowledgement}

The work leading to this article began when the first author was an Early Doctoral Fellow at the Indian Institute of Technology Delhi. The first author gratefully acknowledges the financial support provided by the Indian Institute of Technology Delhi during this period and subsequently by the Indian Institute of Science Education and Research Bhopal through grant ANRF/MAT/2025-2026/154. The second author acknowledges support from the Anusandhan National Research Funding grant ANRF/ARGM/2025/002342/MTR. The third author gratefully acknowledges the Indian Institute of Technology Delhi for providing the Institute Assistantship.

\section*{Data Availability} 
    Data sharing does not apply to this article as no datasets were generated or analyzed during the current study.

\section*{Competing Interests}
    The authors declare that they have no competing interests. 

\bibliographystyle{acm}
\bibliography{ref}

\end{document}